\documentclass[a4paper,oneside]{amsart}
\usepackage[utf8]{inputenc}
\usepackage[left=3cm,right=3cm,top=3cm,bottom=3cm]{geometry}
\usepackage[OT1]{fontenc}
\usepackage{amsmath}
\usepackage{amsfonts}
\usepackage{amssymb}
\usepackage{amsthm}
\usepackage{xypic}
\usepackage[pagebackref]{hyperref}

\DeclareMathOperator{\supp}{supp}
\DeclareMathOperator{\Hom}{Hom}
\DeclareMathOperator{\Tor}{Tor}
\DeclareMathOperator{\Ext}{Ext}
\DeclareMathOperator{\Ker}{Ker}

\DeclareMathOperator{\Img}{Im}
\DeclareMathOperator{\sk}{sk}
\DeclareMathOperator{\hdim}{hdim}
\DeclareMathOperator{\cdim}{cdim}
\DeclareMathOperator{\cd}{cd}
\DeclareMathOperator{\vcd}{vcd}
\DeclareMathOperator{\cat}{cat}
\DeclareMathOperator{\lk}{lk}

\DeclareMathOperator{\chr}{char}

\def\pt{{\mathrm{pt}}}
\def\id{{\mathrm{id}}}

\def\B{\mathrm{B}}
\def\oB{\overline{\B}}
\def\CC{\mathbb{C}}
\def\ZZ{\mathbb{Z}}
\def\FF{\mathbb{F}}
\def\MF{\mathrm{MF}}
\def\QQ{\mathbb{Q}}
\def\RR{\mathbb{R}}

\def\Zg{\ZZ_{\geq 0}}
\def\Zl{\ZZ_{\leq 0}}
\def\Zm{\Zg^m}
\def\RC{\mathrm{RC}}
\def\R{\mathcal{R}}
\def\Z{\mathcal{Z}}
\def\L{\mathcal{L}}
\def\K{\mathcal{K}}
\def\Kf{\K^\mathrm{f}}
\def\ZK{\Z_{\K}}

\def\RK{\R_{\K}}
\def\RCK{\RC_\K}
\def\OZK{\Omega\ZK}
\def\DJ{(\CC\mathrm{P}^\infty)^\K}
\def\ODJ{\Omega\DJ}
\def\k{\mathbf{k}}
\def\H{{\widetilde{H}}}
\def\Lm{\Lambda[m]}
\def\kK{\k[\K]}
\def\kKc{\k\langle\K\rangle}

\newtheorem{thm}{Theorem}[section]
\newtheorem{lmm}[thm]{Lemma}
\newtheorem{prp}[thm]{Proposition}
\newtheorem{crl}[thm]{Corollary}

\theoremstyle{definition}

\newtheorem{dfn}[thm]{Definition}
\newtheorem{rmk}[thm]{Remark}
\newtheorem{exm}[thm]{Example}
\newtheorem{prb}[thm]{Problem}

\numberwithin{equation}{section}

\hypersetup{
    colorlinks=true,
    linkcolor=blue,
    citecolor=blue,
    urlcolor=blue,
    pdftitle={Pontryagin algebras and the LS-category of moment-angle complexes in the flag case},
    pdfauthor={Fedor E. Vylegzhanin},
}

\begin{document}
\author{Fedor E. Vylegzhanin}
\address{\parbox{\linewidth}{
Faculty of Mechanics and Mathematics, Lomonosov Moscow State University, Moscor, 119991 Russia\\
National Research University Higher School of Economics,  Pokrovskii bul. 11, Moscow, 109028 Russia}}
\email{vylegf@gmail.com}
\title[Pontryagin algebras and the LS-category]{Pontryagin algebras and the LS-category\\of moment-angle complexes in the flag case}


\begin{abstract}
For any flag simplicial complex $\K,$ we describe the multigraded Poincar\'e series, the minimal number of relations, and the degrees of these relations in
the Pontryagin algebra of the corresponding moment-angle complex $\ZK.$ 
We compute the LS-category of $\ZK$ for flag complexes and give a lower bound in the general case. The key observation is that the Milnor-Moore spectral sequence collapses at the second sheet for flag $\K.$

We also show that the results of Panov and Ray about the Pontryagin algebras of Davis-Januszkiewicz spaces are valid for arbitrary coefficient rings, and introduce the $(\ZZ\times\Zm)$-grading on the Pontryagin algebras which is similar to the multigrading on the cohomology of $\ZK.$
\end{abstract}

\maketitle
\section{Introduction}
Let $\K$ be a simplicial complex on the vertex set $[m]=\{1,\dots,m\}.$ We consider the following homotopy invariants of the \emph{moment-angle complex} $\ZK$ associated with $\K:$
\begin{itemize}
\item The \emph{Pontryagin algebra} $H_*(\OZK;\k),$ where $\k$ is a commutative ring with unit;
\item The \emph{Lusternik-Schnirelmann category} (LS-category) $\cat(\ZK)$ of $\ZK.$
\end{itemize}

The Pontryagin algebras of moment-angle complexes were studied in \cite{pr,gptw,onerelator}. If $\k$ is a field and $\K$ is a flag complex, a minimal set of $\sum_{J\subset[m]}\dim\H_0(\K_J)$ generators for $H_*(\OZK;\k)$ is known (see \cite[Theorem 4.3]{gptw}). However, it seems difficult to explicitly describe the relations between them. We solve a simpler problem of describing the \emph{minimal number} of defining relations and their \emph{degrees} by computing $\Tor_2^{H_*(\OZK;\k)}(\k,\k).$
(Assertion 2 of Proposition \ref{prp:homological} below shows why this $\k$-module is relevant.)

Beben and Grbi\'c \cite{beben_grbic} gave various lower and upper bounds on the LS-category of moment-angle complexes, focusing on the case where $|\K|$ is a low-dimensional manifold (but $\K$ can be non-flag). We compute $\cat(\ZK)$ for all flag complexes $\K$ and give a new lower bound in the general case. These results also use our calculation of $\Tor^{H_*(\OZK;\k)}(\k,\k)$ for flag complexes.

For any simplicial complex $\K,$ there is an embedding of $H_*(\OZK;\k)$ into the Pontryagin algebra $H_*(\ODJ;\k)$ of the \emph{Davis-Januszkiewicz space} $\DJ.$
We introduce the $(\ZZ\times\Zm)$-grading on these algebras, and show that Panov and Ray's theorem on the structure of $H_*(\ODJ;\k)$ is true for any commutative ring $\k$ with unit (for the case of integer or field coefficients, see \cite[Sect. 8.4]{ToricTopology}).
\begin{thm}
\label{thm:hodj_description}
Consider the $(\ZZ\times\Zm)$-graded $\k$-algebra
$$\kK^!:=T(u_1,\dots,u_m)/(u_i^2=0,~i=1,\dots,m;~u_iu_j+u_ju_i=0,~\{i,j\}\in\K),~\deg u_i=(-1,2e_i).$$ Then
\begin{enumerate}
    \item[(a)] there is an isomorphism $H_*(\ODJ;\k)\cong\Ext_{\kK}(\k,\k)$ of $\ZZ$-graded algebras, where $\kK$ is the Stanley-Reisner algebra; in particular, for any $n\geq 0$ we have
    $$H_n(\ODJ;\k)\cong\bigoplus_{-i+2|\alpha|=n}\Ext^i_{\kK}(\k,\k)_{2\alpha},\quad i\geq 0,~\alpha\in\Zm;$$
    \item[(b)] there is a natural inclusion $\kK^!\hookrightarrow H_*(\ODJ;\k);$
    \item[(c)] if $\K$ is a flag complex, then $\kK^!\cong H_*(\ODJ;\k).$
\end{enumerate}
\end{thm}
The isomorphism from (a) induces the following $(\ZZ\times\Zm)$-grading on $H_*(\ODJ;\k):$
$$H_n(\ODJ;\k)=\bigoplus_{-i+2|\alpha|=n}H_{-i,2\alpha}(\ODJ;\k),\quad H_{-i,2\alpha}(\ODJ;\k)\cong\Ext^{i}_{\kK}(\k,\k)_{2\alpha}.$$
The inclusion from (b) agrees with this grading.
Note that it is similar to the multigrading on $H^*(\ZK;\k)$ of \cite[Theorem 4.5.7]{ToricTopology}:
$$H^n(\ZK;\k)=\bigoplus_{-i+2|\alpha|=n}H^{-i,2\alpha}(\ZK;\k),\quad H^{-i,2\alpha}(\ZK;\k)\cong\Tor^{\k[v_1,\dots,v_m]}_i(\kK,\k)_{2\alpha}.$$

We also obtain a $(\ZZ\times\Zm)$-grading on $H_*(\OZK;\k),$ since it is a subalgebra of $H_*(\ODJ;\k).$ Using the Fr\"oberg resolution \cite{froberg} for the left $\kK^!$-module $\k$  
and the structure of free left $H_*(\OZK;\k)$-module on $H_*(\ODJ;\k),$
we compute $\Tor_n^{H_*(\OZK;\k)}(\k,\k).$

\begin{thm}
\label{thm:tor_ozk_answer}
Let $\K$  be a flag simplicial complex. Then $$\Tor_n^{H_*(\OZK;\k)}(\k,\k)\cong \bigoplus_{J\subset[m]}\widetilde{H}_{n-1}(\K_J;\k).$$
Here $\K_J=\{I\in\K:~I\subset J\}$ is the full subcomplex of $\K.$
This isomorphism is $(\ZZ\times\Zm)$-graded in the following sense:
$$\Tor_n^{H_*(\OZK;\k)}(\k,\k)_{-|J|,2J}\cong \widetilde{H}_{n-1}(\K_J;\k),\quad J\subset[m],$$
all the other graded components being zero.

(For a subset $J\subset[m],$ we write $J$ instead of $\sum_{j\in J}e_j\in \Zm.$)
\end{thm}
Using this result, in Corollary \ref{crl:number_of_gen_and_rel} we compute the number of generators and relations in the algebra $H_*(\OZK;\k)$ for field coefficients, and give a lower bound in the case where the coefficient ring is a principal ideal domain (PID). This implies the available criteria for $H_*(\OZK;\FF)$ to be free or be a one-relator algebra (see \cite{gptw,onerelator}).

The \emph{Poincar\'e series} of $H_*(\OZK;\FF)$ in the flag case was calculated by Panov and Ray \cite{pr}. We refine this calculation in Theorem \ref{thm:poincare_series} by introducing the multigrading.

By the Milnor-Moore theorem~\cite{milnor_moore}, the algebra $H_*(\OZK;\QQ)$ is the universal enveloping of $\pi_*(\OZK)\otimes\QQ.$ Hence the rational homotopy groups of $\ZK$ are $(\ZZ\times\Zm)$-graded.
We obtain a multigraded refinement of \cite[Theorem~4.2.1]{denham_suciu} by calculating the multigraded Poincar\'e series of $\pi_*(\OZK)\otimes\QQ$ (see Theorem \ref{thm:homotopy_ranks}). The non-negativity of the graded components of these series gives certain restrictions on the Euler characteristics of full subcomplexes of $\K$ in the flag case. These restrictions (see Example \ref{exm:restriction_on_chi}) can be of independent combinatorial interest (see \cite{ustinovskiy}).

~

As another application of Theorem~\ref{thm:tor_ozk_answer}, we prove in Proposition \ref{prp:mm_collapse} that the \emph{Milnor-Moore spectral sequence}
$$E^2_{p,q}\cong\Tor^{H_*(\Omega X;\FF)}_p(\FF,\FF)_q\Rightarrow H_{p+q}(X;\FF)$$
collapses at the second page for $X=\ZK$ in the flag case. This allows us to compute the \emph{Toomer invariant} \cite{toomer} of such $\ZK,$ which gives a lower bound on the LS-category $\cat(\ZK)$ (see Proposition \ref{prp:cat_zk_lower_bound_flag}). Note that the simpler lower bound of $\cat(X)$ by the \emph{cup-length} is not always strict for $\ZK,$ as shown in \cite[Sect. 5]{beben_grbic}. However, we expect it to be strict for flag complexes (see Problem \ref{prb:cup_for_flag}).

To give an upper bound, we observe that in the flag case the \emph{real moment-angle complex} $\RK$ is the classifying space for the commutator subgroup $\RCK'$ of the \emph{right-angled Coxeter group} $\RCK$ associated with $\K$ (see \cite{pv}).
Using the inequality $\cat(\ZK)\leq\cat(\RK)$ by Beben and Grbi\'c \cite{beben_grbic} and Dranishnikov's formula \cite{dranishnikov} for the virtual cohomological dimension of $\RCK,$ we compute $\cat(\ZK)$ for every flag $\K.$
\begin{thm}
\label{thm:cat_final}
Let $\K$ be a flag simplicial complex. Then $$\cat(\ZK)=\cat(\RK)= 1+\max_{J\subset[m]}\cdim_\ZZ \K_J=1+\max_{I\in\K}\cdim_\ZZ\lk_\K I.$$
\end{thm}

Let $\Kf$ be the \emph{flagification} of $\K,$ the unique flag simplicial complex with the same $1$-skeleton. (The \emph{$i$-skeleton} of $\K$ is the simplicial complex $\sk_i\K:=\{I\in\K:~|I|\leq i+1\}.$) We give the following lower bound on $\cat(\ZK)$ in the non-flag case, which is most useful if $\nu(\K)$ is small (this number measures how far $\K$ is from being flag).

\begin{prp}
\label{prp:cat_zk_lower_bound}
Let $\K$ be a simplicial complex. Let $\nu(\K)$ be the smallest $n\geq 0$ such that the following holds: $J\in\K$ whenever $J\subset I\in\Kf$ and $|I\setminus J|\geq n.$ Then $$\cat(\ZK)\geq 1-\nu(\K)+\max_{J\subset[m]}\cdim_\ZZ\Kf_J.$$
\end{prp}

As an example, we calculate $\cat(\ZK)$ for the skeleta of flag triangulations of manifolds (see Corollary \ref{crl:cat_sk_manifold}).

The paper is organised as follows.

In Section \ref{section:preliminaries} we recall the basic properties of simplicial complexes and moment-angle complexes. Then we discuss multigraded associative algebras and relevant homological algebra.

In Section \ref{section:hodj_description} we prove Theorem \ref{thm:hodj_description} and describe the structure of a left $H_*(\OZK;\k)$-module on $H_*(\ODJ;\k).$

In Section \ref{section:relations_algebras} we interpret the minimal resolution of the left $H_*(\ODJ;\k)$-module $\k$ constructed by Fr\"oberg as a free resolution of the left $H_*(\OZK;\k)$-module $\k.$ This resolution is used to prove Theorems \ref{thm:tor_ozk_answer}, \ref{thm:poincare_series}, and \ref{thm:homotopy_ranks}.

In Section \ref{section:ls} we study the LS-category of $\ZK$ and $\RK$ in the flag case. We first calculate $\cat(\RK)$ and then show that the inequality $\cat(\ZK)\leq\cat(\RK)$ turns into an equality due to the collapse of the Milnor-Moore spectral sequence. Finally, we discuss lower bounds on $\cat(\ZK)$ in the non-flag case and the cup-length of $\ZK$ in the flag case.

\section{Preliminaries}
\label{section:preliminaries}

\subsection{Simplicial complexes and polyhedral products}
A \emph{simplicial complex} $\K$ on a vertex set $V$ is a non-empty collection of subsets $I\subset V,$ called \emph{faces}, that satisfies the following condition:
\begin{itemize}
\item If $I\in\K$ and $J\subset I,$ then $J\in\K.$
\end{itemize}
It follows that $\varnothing\in\K.$ An element $i\in V$ is called a \emph{ghost vertex} if $\{i\}\notin\K.$ We consider only complexes without ghost vertices, so
$\{i\}\in\K$ for every $i\in V.$
Usually $V=[m]:=\{1,\dots,m\}.$

For any $J\subset[m],$ the simplicial complex $\K_J:=\{I\in\K:~I\subset J\}$ on the vertex set $J$ is called a \emph{full subcomplex} of $\K$ on $J.$ 

A \emph{missing face} of $\K$ is a set $J\subset [m]$ such that $J\notin\K$,but every proper subset of $J$ is a face of $\K.$ A complex $\K$ is a \emph{flag} complex if all of its missing faces consist of two vertices. Every full subcomplex of a flag complex is a flag complex. A flag complex is uniquely determined by its 1-skeleton. Hence, for every $\K,$ there is a unique flag complex $\Kf$ with the same $1$-skeleton, called the \emph{flagification} of $\K.$ Clearly, $\K\subset\Kf.$

The \emph{link} of a simplex $I\in\K$ is the simplicial complex $\lk_\K I :=\{J\in\K: I\cap J=\varnothing,~I\cup J\in\K\}.$

\begin{lmm}[{\cite[Lemma 2.3]{pt}}]
\label{lmm:lk_is_full_subcomplex}
Let $\K$ be a flag simplicial complex. Let $I\in\K.$ Then $\lk_\K I $ is a full subcomplex of $\K.$\qed
\end{lmm}

Now let $(\underline{X},\underline{A})=\{(X_i,A_i)\}_{i=1}^m$ be a sequence of pairs of topological spaces and $\K$ be a simplicial complex on $[m].$ The corresponding \emph{polyhedral product} is the topological space
$$(\underline{X},\underline{A})^\K:=\bigcup_{I\in \K}\left(\prod_{i\in I}X_i\times \prod_{i\notin I}A_i\right) \subset \prod_{i=1}^m X_i.$$
We write $(X,A)^\K:=(\underline{X},\underline{A})^\K$ if $X_1=\dots=X_m=X,$ $A_1=\dots=A_m=A.$ We also set $X^\K:=(X,\pt)^\K.$

The
\emph{moment-angle complex} $\ZK:=(D^2,S^1)^\K$ and \emph{real moment-angle complex} $\RK:=(D^1,S^0)^\K$ are important special cases of this construction. The homology and cohomology of moment-angle complexes are well known.
\begin{thm}[see {\cite[Theorem~4.5.8]{ToricTopology}}]\label{thm:moment_angle_homology}\pushQED{\qed}

For any abelian group of coefficients,
$$
    H_p(\ZK)\cong\bigoplus_{J\subset[m]}\H_{p-|J|-1}(\K_J),\quad H_p(\RK)\cong\bigoplus_{J\subset[m]} \H_{p-1}(\K_J),
$$
\[
    H^p(\ZK)\cong\bigoplus_{J\subset[m]}\H^{p-|J|-1}(\K_J),\quad H^p(\RK)\cong\bigoplus_{J\subset[m]} \H^{p-1}(\K_J).\qedhere
\]
%\popQED
\end{thm}

\subsection{Multigraded associative algebras}
In what follows, $\k$ is a commutative ring with unit and $\otimes$ is the tensor product of $\k$-modules. For field coefficients, we usually write $\FF$ instead of $\k.$

By a \emph{multigrading} of a $\k$-module we mean a grading by an additive semigroup of the form $\ZZ^k\times\Zm,$ $k=0,1,2.$ The multidegree of a homogeneous element $x$ is denoted by $\deg x$ or by $|x|.$ We use Greek letters for elements of $\Zm$ and $k$-tuples of latin letters for elements of  $\ZZ^k$. The basis of $\Zm$ is denoted by $e_1,\dots,e_m.$ For a multiindex $\alpha=(\alpha_1,\dots,\alpha_m)\in\Zm$ we write $|\alpha|:=\sum_{i=1}^m \alpha_i.$

There are standard projections $\Zm\to\ZZ,~e_i\mapsto 1$ and $\ZZ^k\to\ZZ,~(n_1,\dots,n_k)\mapsto n_1+\dots+n_k.$ Hence every $(\ZZ^k\times\Zm)$-graded $\k$-module is $\ZZ$-graded by the total grading:
$$V_n:=\bigoplus_{i_1+\dots+i_k+|\alpha|=n}V_{i_1,\dots,i_k,\alpha}.$$% Sometimes we write $|x|$ instead of $\deg x.$

For the degree of homogeneous components, we use the agreement $\deg M_\alpha=\deg M^\alpha=\alpha$ with the only exception: $\deg\Ext^n_A = -n.$ For a free graded $\k$-module   $M,$ let $M^\#$ denote the \emph{graded dual} $\k$-module: $(M^\#)_\alpha:=\Hom_\k(M_\alpha,\k).$

By a \emph{multigraded algebra} we mean a $(\ZZ^k\times\Zm)$-graded connected associative $\k$-algebra with unit which is a free $\k$-module of finite type. (A $(\ZZ^k\times\Zm)$-graded algebra $A$ is \emph{connected} if it is connected as a $\ZZ$-graded algebra, i.e. $A_{<0}=0$ and $A_0\cong\k.$) Here are some examples:
\begin{enumerate}
\item The $(\ZZ^k\times\Zm)$-graded free associative algebra (tensor algebra) $T(a_1,\dots,a_N)$ for arbitrary generators $a_i$ of positive total degree;
\item The $\Zm$-graded commutative $\k$-algebra $\k[m]:=\k[v_1,\dots,v_m],$ $\deg v_i=2e_i;$
\item The $\Zm$-graded \emph{Stanley-Reisner algebra} of a simplicial complex $\K,$
$$\textstyle\kK:=\k[v_1,\dots,v_m]/\left(\prod_{i\in I}v_i: I\notin\K\right),\quad\deg v_i=2e_i;$$
\item The $(\ZZ\times\Zm)$-graded exterior $\k$-algebra $\Lm:=\Lambda[u_1,\dots,u_m],$ $\deg u_i=(-1,2e_i).$
\end{enumerate}

Now let $A$ be a $(\ZZ^k\times\Zm)$-graded $\k$-algebra, $L,\widetilde{L}$ be left $A$-modules and $R$ be a right $A$-module. Then, for every $i\geq 0,$ the $\k$-modules $\Ext^i_A(L,\widetilde{L})$ and $\Tor^A_i(R,L)$ are $(\ZZ^k\times\Zm)$-graded. Hence $\Ext_A(L,\widetilde{L})$ and $\Tor^A(L,R)$ are $(\ZZ\times\ZZ^k\times\Zm)$-graded $\k$-modules. For example,
\begin{align*}
\text{if }x\in\Ext_A^i(L,\widetilde{L})_{j_1,\dots,j_k,\alpha},&
\quad\text{ then }\deg x=(-i,j_1,\dots,j_k,\alpha);\\
\text{if }y\in\Tor^A_i(R,L)_{j_1,\dots,j_k,\alpha},&
\quad\text{ then }\deg y=(i,j_1,\dots,j_k,\alpha).
\end{align*}

The isomorphism $\Lm\cong\Ext_{\k[m]}(\k,\k)$ explains our choice of grading for $\Lm.$

\subsection{Poincar\'e series and minimal presentations}
\begin{dfn}
Let $V=\bigoplus_{i,\alpha} V_{i,\alpha}$ be a $(\ZZ\times\Zm)$-graded vector space. The \emph{Poincar\'e series} of $V$ is a formal power series on $m+1$ variables $(t,\lambda)=(t,\lambda_1,\dots,\lambda_m):$
$$F(V;t,\lambda):=\sum_{i\in\ZZ}\sum_{\alpha\in \Zm} \dim(V_{i,\alpha})\cdot t^i\lambda^\alpha,\quad\lambda^\alpha:=\prod_{i=1}^m \lambda_i^{\alpha_i}.$$
\end{dfn}
The space $V$ will usually be concentrated in degrees $\Zl\times\Zm,$ so $F(V;t,\lambda)\in\ZZ[[t^{-1},\lambda]].$

A \emph{presentation} of a multigraded associative $\k$-algebra $A$ with unit is an epimorphism $\pi:T(a_1,\dots,a_N)\twoheadrightarrow A$ with a choice of elements $r_1,\dots,r_M\in T(a_1,\dots,a_N)$ called \emph{relations} that generate the ideal $\Ker\pi\subset T(a_1,\dots,a_N).$
We assume that the generators have positive degrees with respect to the total grading, and that the relations are homogeneous. Hence $A$ is connected, and the projection $\varepsilon:A\to\k$ makes $\k$ into a left $A$-module. Moreover, the following standard properties of connected graded algebras remain valid, with the same proofs by induction.

\begin{prp}
\label{prp:homological}
Let $\FF$ be a field, and let $A$ be a multigraded $\FF$-algebra.
\begin{enumerate}
\item There is a free resolution of the left $A$-module $\FF$ of the form
$$\xymatrix{
\dots\ar[r]^-{d}& A\otimes_\FF R_2\ar[r]^-{d} & A\otimes_\FF R_1\ar[r]^-{d} & A\ar[r]^-{\varepsilon} & \FF\ar[r] & 0
}$$
such that $d:R_n\to A^+\otimes_\FF R_{n-1}$ and $d:R_1\to A^+.$
This resolution is called a \emph{minimal} resolution and is unique up to an isomorphism. After applying the functor $\FF\otimes_A(-),$ all differentials become zero; hence $R_n\cong\Tor^A_n(\FF,\FF).$ 

\item Up to a natural equivalence, there is a one-to-one correspondence between presentations $A\simeq T(a_1,\dots,a_N)/(r_1,\dots,r_M)$ and exact sequences of left $A$-modules of the form
$A\otimes_\FF R_2\to A\otimes_\FF R_1\to A\overset\varepsilon\longrightarrow \FF\to 0.$
Under this correspondence, $\{a_1,\dots,a_N\}$ is a basis of the graded vector space $R_1,$ and $\{r_1,\dots,r_M\}$ is a basis of $R_2.$

In particular, there is a \emph{minimal presentation} $A\simeq T(a_1,\dots,a_N)/(r_1,\dots,r_M)$ such that $$\Tor^A_1(\FF,\FF)\simeq \bigoplus_{i=1}^N \FF\cdot a_i,\quad \Tor^A_2(\FF,\FF)\simeq\bigoplus_{j=1}^M\FF\cdot r_j.$$
For a minimal presentation, the numbers of generators and relations in each degree are the smallest possible among all presentations of $A.$
\item The Poincar\'e series of $A$ and $\Tor^A_n(\FF,\FF)$ are related by the identity
$$\sum_{n=0}^\infty (-1)^n\cdot F(\Tor_n^A(\FF,\FF);t,\lambda)=\frac{1}{F(A;t,\lambda)}.$$
\end{enumerate}
\end{prp}
\begin{proof}
\begin{enumerate}
\item See \cite[Proposition A.2.3]{ToricTopology}.
\item See \cite[\S 7]{wall}.
\item See \cite[Proposition A.2.1]{ToricTopology}.
\qedhere
\end{enumerate}
\end{proof}

For an arbitrary coefficient ring, the differentials in the complex obtained after applying the functor $\k\otimes_A (-)$ to the minimal resolution can be non-zero (see the remark after Proposition A.2.3 in \cite{ToricTopology}). However, any presentation of $A$ still corresponds to a resolution, and $\Tor^A(\k,\k)$ is the homology of the complex $\dots\to\bigoplus\k\cdot r_j\to\bigoplus \k\cdot a_i\to \k\to 0.$ This gives a lower bound on the number of generators and relations.
\begin{prp}
\label{prp:homological_for_rings}
Let $\k$ be a PID and $A$ be a multigraded $\k$-algebra. Denote by $N_{i,\alpha}$ the minimal number of generators of the $\k$-module $\Tor^A_i(\k,\k)_\alpha.$ Then every presentation of $A$ has at least $N_{1,\alpha}$ generators and $N_{2,\alpha}$ relations of degree $\alpha.$
\end{prp}
\begin{proof}
Let $A=T(a_1,\dots,a_N)/(r_1,\dots,r_M)$ be a presentation of $A.$ We construct the exact sequence
$$\bigoplus_{j=1}^M A\cdot r_j\to \bigoplus_{i=1}^N A\cdot a_i\to A\to \k\to 0$$
as in \cite[\S 7]{wall}, and extend it to obtain a free resolution of the left $A$-module $\k.$ Then $\Tor^A(\k,\k)$ is the homology of the complex
$$\dots\to\bigoplus_{j=1}^M \k\cdot r_j\to\bigoplus_{i=1}^N \k\cdot a_i\to \k\to 0.$$
In particular, $\Tor_2^A(\k,\k)_\alpha$ is a
quotient module of a submodule of the free $\k$-module generated by $\{r_j: \deg r_j=\alpha\}.$ For PIDs, every submodule of a free module is free with at most the same rank (see \cite[Proposition 5.1]{aluffi}). So $\Tor^A_2(\k,\k)_\alpha$ can be generated by $\#\{r_j:\deg r_j=\alpha\}$ elements, and this implies that $\#\{r_j:\deg r_j=\alpha\}\geq N_{2,\alpha}.$ Similarly, $\#\{a_i:\deg a_i=\alpha\}\geq N_{1,\alpha}.$
\end{proof}

\subsection{Stanley-Reisner (co)algebras}
Recall that the \emph{Stanley-Reisner algebra} is the following $\Zm$-graded $\k$-algebra:
$$\textstyle\kK:=\k[v_1,\dots,v_m]/\left(\prod_{i\in I}v_i: I\notin\K\right),\quad\deg v_i=2e_i.$$
For $\alpha=(\alpha_1,\dots,\alpha_m)\in\ZZ_{\ge 0}^m,$ we define $\supp\alpha:=\{j\in [m]: \alpha_j>0\}$ and $v^\alpha:=\prod_{j=1}^m v_j^{\alpha_j}\in \kK.$
Then $\{v^\alpha\}_{\supp\alpha\in \K}$ is an additive basis of $\kK$ and 
$$v^\alpha\cdot v^\beta=\begin{cases}v^{\alpha+\beta},&\supp\alpha\cup\supp\beta\in\K,\\ 0,&\supp\alpha\cup \supp\beta\notin\K.\end{cases}$$ 
We also consider the graded dual coalgebra $\kKc:=\kK^\#$ with the dual basis $\{\chi_\alpha\}_{\supp\alpha\in\K},$ $\chi_\alpha:=(v^\alpha)^\#,$ $\deg\chi_\alpha=2\alpha.$ The comultiplication in $\kKc$ is given by
$$\Delta \chi_\alpha=\sum_{\alpha=\beta+\gamma}\chi_\beta\otimes \chi_\gamma,$$ where the sum is taken over all $\beta,\gamma\in\Zm.$
In particular,
\begin{align*}
\Delta \chi_{ii}&=1\otimes \chi_{ii}+\chi_i\otimes \chi_i+\chi_{ii}\otimes 1,\quad &i=1,\dots,m;\\
\Delta \chi_{ij}&=1\otimes \chi_{ij}+\chi_i\otimes \chi_j+\chi_j\otimes \chi_i+\chi_{ij}\otimes 1,\quad &\{i,j\}\in\K.
\end{align*}

We will need the following notation: $\kK_{(s)}$ is the $\k$-submodule of $\kK$ with additive basis $$\{v^\alpha\in\kK:|\alpha|=s\}.$$ The modules $\kKc_{(s)}$ are defined similarly. For example, $\chi_{ij}\in\kKc_{(2)}$ while $\deg\chi_{ij}=2e_i+2e_j.$
\subsection{The bar and cobar constructions}
\label{subsection:bar}
We follow \cite[\S1]{priddy} (see also \cite[\S7.1, \S9.2]{mccleary}). Let $A$ be a connected $\Zm$-graded $\k$-algebra with unit which is a free $\k$-module of finite type. Let $I:=\Ker(\varepsilon:A\to\k)=A_{>0}\subset A$ be the augmentation ideal. We also set $\overline{a}:=(-1)^{1+|a|}\cdot a.$ Then, for a graded left $A$-module $L$ and a graded right $A$-module $R,$ the \emph{two-sided bar construction} $\B(R,L):=R\otimes T(I)\otimes L$ is the $(\ZZ\times\Zm)$-graded complex with grading
$$\deg r[a_1|\dots|a_s]l:= (s, |r|+\sum_{i=1}^s |a_i|+|l|)$$ and the following differential $\partial$ of degree $(-1,0):$
$$-\partial(r[a_1|\dots|a_s]l):=\overline{r}a_1[a_2|\dots|a_s]l+\sum_{i=1}^{s-1}\overline{r}[\overline{a}_1|\dots|\overline{a}_ia_{i+1}|\dots|a_s]l+\overline{r}[\overline{a}_1|\dots|\overline{a}_{s-1}]a_sl.$$
Here $r[a_1|\dots|a_s]l$ is the traditional notation for the element $$r\otimes (a_1\otimes\dots\otimes a_s)\otimes l\in R\otimes T(I)\otimes L= \B(R,L).$$

It is well known that $\B(A,L)$ is a free resolution of a left $A$-module $L.$ Hence $\Ext_A(\k,\k)$ can be computed as the cohomology of the complex
$$\Hom_A(\B(A,\k),\k)\cong  (\k\otimes_A\B(A,\k))^\#\cong \B(\k,\k)^\#.$$
(Here we use the \emph{adjunction isomorphism} $\Hom_A(M,N)\cong (N^\#\otimes_A M)^\#$ for left $A$-modules $M,N$ which are free $\k$-modules of finite type.)

This complex coincides with the \emph{Adams cobar construction} $\Omega_*A^{\#}$ (see \cite{adams}), where the coalgebra $A^\#:=\Hom_\k(A,\k)$ is the graded dual of $A.$ The cobar construction is the $(\ZZ\times\Zm)$-graded $\k$-module $\Omega_*A^\#\cong T(I^\#)$ with grading
$$\deg [x_1|\dots|x_s]:=(-s,|x_1|+\dots+|x_s|)$$ and the following differential $\partial^\#$ of degree $(-1,0):$
if $\Delta x_i=1\otimes x_i+x_i\otimes 1+\sum_r x_{i,r}'\otimes x_{i,r}''$ then
$$\partial^\#[x_1|\dots|x_s]=\sum_{i=1}^{s}\sum_r [\overline{x}_1|\dots|\overline{x}_{i-1}|\overline{x}_{i,r}'|x_{i,r}''|x_{i+1}|\dots|x_s].$$

The cobar construction is a differential graded algebra with the multiplication
$$[x_1|\dots|x_s]\smile [y_1|\dots|y_t]:=[x_1|\dots|x_s|y_1|\dots|y_t].$$
Hence the $(\ZZ\times\Zm)$-graded $\k$-module $H(\Omega_* A^\#)$ is an associative algebra.

We summarize this as follows.
\begin{prp}
\label{prp:bar-cobar}
Let $\k$ be a commutative ring with unit. Let $A$ be a connected $\Zm$-graded $\k$-algebra with unit which is a free $\k$-module of finite type. Then there is an isomorphism $\Ext_A(\k,\k)\cong H(\Omega_* A^\#)$ of $(\ZZ\times\Zm)$-graded $\k$-modules, and this isomorphism makes $\Ext_A(\k,\k)$ into a connected graded associative algebra.\qed
\end{prp}
\section{Pontryagin algebras of Davis-Januszkiewicz spaces}
\label{section:hodj_description}

\subsection{Proof of Theorem \ref{thm:hodj_description}}
\subsubsection*{Proof of assertion \emph{(a)}}
By \cite[Proposition 8.4.10]{ToricTopology}, we have an isomorphism $H_*(\ODJ;\k)\cong H(\Omega_*\kKc)$ of graded $\k$-algebras.
Since $\kK=(\kKc)^\#$ is a free $\k$-module of finite type, it follows from Proposition \ref{prp:bar-cobar} that $H(\Omega_*\kKc)\cong \Ext_{\kK}(\k,\k)$ as graded $\k$-algebras.
\subsubsection*{Proof of assertion \emph{(b)}}
We will follow \cite[Corollary 1.1]{lofwall} to prove that $\kK^!$ is isomorphic to the ``diagonal'' subalgebra
$$D=\bigoplus_{i,\alpha:~|\alpha|=i}\Ext^i_{\kK}(\k,\k)_{2\alpha}\subset \Ext_{\kK}(\k,\k).$$
Define
$$D_0:=\bigoplus_{i,\alpha:~|\alpha|=i}\Omega_{-i}\kKc_{2\alpha}\subset \Omega_*\kKc.$$
Every element of $D_0$ is of the form $[\chi_{j_1}|\dots|\chi_{j_s}]$ by dimension reasons, so $D_0\cong T(\chi_1,\dots,\chi_m).$

Since the cobar construction $(\Omega_*\kKc,\partial^\#)$ is concentrated in degrees
$$\{(-i,2\alpha):~0\leq i\leq |\alpha|\}\subset\ZZ\times\Zm$$
and $\partial^\#$ has degree $(-1,0),$ it follows that $D= D_0/(D_0\cap \Img\partial^\#).$ Note that the $\k$-submodule $D_0\cap\Img\partial^\#\subset D_0$ is a two-sided ideal, since $D$ is an algebra. Now let ]
$$\varphi^\#:\kKc_{(2)}\to \kKc_{(1)}\otimes\kKc_{(1)}$$
be the non-primitive part of the comultiplication map. Since $\varphi^\#(\chi_{ii})=\chi_i\otimes \chi_i$ for any $i=1,\dots,m$ and $\varphi^\#(\chi_{ij})=\chi_i\otimes \chi_j+\chi_j\otimes \chi_i$ for any $\{i,j\}\in\K,$ it follows that $\kK^!\cong D_0/(\Img\varphi^\#).$

Finally, we show that $D_0\cap \Img\partial^\#= (\Img\varphi^\#)$ as two-sided ideals in $D_0.$ Let $f\in D_0\cap\Img\partial^\#.$ Since $\partial^\#$ has degree $(-1,0),$ we can assume that $f$ is homogeneous of degree $(-i,2\alpha),$ $|\alpha|=i.$ Then $f=\partial^\#g$ for some $g,$ $\deg g=(-i+1,2\alpha).$ Let $V=\kKc_{(1)}$ and $W=\kKc_{(2)}.$ We have
$$g\in \Omega_{-i+1}\kKc_{2\alpha}=\bigoplus_{i=0}^{n-2}V^{\otimes i}\otimes W\otimes V^{\otimes (n-i-2)}$$
for dimensional reasons. Since $\partial^\#V=0$ and $\partial^\#|_W=\varphi^\#,$ we get
$$\partial^\#g\in\sum_{i=0}^{n-2}V^{\otimes i}\otimes \Img\varphi^\#\otimes V^{\otimes (n-i-2)} \subset D_0\cdot \Img\varphi^\#\cdot D_0,$$ so $D_0\cap\Img\partial^\#\subset (\Img\varphi^\#).$
Conversely, $\partial^\#=\varphi^\#$ when restricted to $\kKc_{(2)}\subset I^\#\subset\Omega_*\kKc;$ hence $\Img\varphi^\#\subset D_0\cap\Img\partial^\#.$ 

Thus,
$$\Ext_{\kK}(\k,\k)\supset D=D_0/(D_0\cap\Img\partial^\#)=D_0/(\Img\varphi^\#)\cong\kK^!.$$

\subsubsection*{Proof of assertion \emph{(c)}}
If $\K$ is a flag simplicial complex, then $\kK$ is a quadratic algebra:
$$\kK=T(v_1,\dots,v_m)/(v_iv_j=0,~\{i,j\}\notin\K;~v_iv_j=v_jv_i,~\{i,j\}\in\K).$$

Hence the algebra $\kK\otimes\kK^!$ belongs to the class of quadratic algebras considered by Fr\"oberg in \cite{froberg}: our variables $v_i$ and $u_i$ correspond to his $X_i$ and $Y_i.$

Since condition B' of \cite[Lemma 4]{froberg} is satisfied, there is a free resolution $(\kK\otimes\kK^!,d)$ for the left $\kK$-module $\k:$
\begin{equation}
\label{eqn:face_ring_resolution}
\dots\overset d\longrightarrow
\kK\otimes \kK^!_{(2)}\overset d\longrightarrow
\kK\otimes \kK^!_{(1)}\overset d\longrightarrow
\kK\otimes \kK^!_{(0)}
\overset\varepsilon\longrightarrow\k\to 0.
\end{equation}

This complex is $(\ZZ\times\Zm)$-graded:
$$\deg u_i=(-1,2e_i),\quad\deg v_i=(0,2e_i).$$
Then $d$ has degree $(1,0),$ and $H^{*}((\kK^!)^\#_*,d^\#)\cong \Ext^{*}_{\kK}(\k,\k)_*.$ It follows from the grading on $\kK^!$ that $\Ext_{\kK}^i(\k,\k)_{2\alpha}=0$ unless $i=|\alpha|.$ Hence $\Ext_{\kK}(\k,\k)=D.$\qed
\begin{rmk}
We used the results stated in
\cite{froberg, lofwall, priddy}
in the case of field coefficients. However, the proofs remain valid for arbitrary $\k,$ provided that the $\k$-modules are free and of finite type. For example, the proofs use the exactness of the functor $\Hom_\k(-,\k)$ for free $\k$-modules and the adjunction isomorphisms $M^{\#\#}\cong M$ and $\Hom_A(M,N^\#)\cong (N\otimes_A M)^\#,$ but never use the universal coefficient theorem.
\end{rmk}
\begin{rmk}
\label{rmk:panov_theriault_argument}
It would be interesting to deduce Theorem \ref{thm:hodj_description}(b) from Theorem \ref{thm:hodj_description}(c) topologically, using the flagification. By Panov and Theriault's result \cite{pt}, the map $\ODJ\to \Omega(\CC\mathrm{P}^\infty)^{\Kf}$ has a right homotopy inverse. Passing to homology, we obtain a $\k$-linear section $\kK^!=\k[\Kf]^!\hookrightarrow H_*(\ODJ;\k)$ of the $(\ZZ\times\Zm)$-graded multiplicative map $H_*(\ODJ;\k)\to\kK^!.$ However, it is unclear if the section is multiplicative and $(\ZZ\times\Zm)$-graded.
\end{rmk}
\begin{rmk}
We see from Theorem \ref{thm:hodj_description}(c) that, in the flag case, the algebra $H_*(\ODJ;\k)$ is concentrated in degrees $\{(-i,2\alpha)\in\ZZ\times\Zm:~i\geq 0,~i=|\alpha|\}.$ If $\K$ is not flag, it is not the case. (We can only say that it is concentrated in degrees $\{(-i,2\alpha)\in\ZZ\times\Zm:~0\leq i\leq|\alpha|\}.$)

For example, it can be shown that every missing face $J\notin\K$ corresponds to a certain nonzero element $\mu_J\in \Ext^2_{\kK}(\k,\k)_{2J}\cong H_{-2,2J}(\ODJ;\k)$ which can be interpreted as a ``higher commutator product'' $\mu_J=[\mu_{j_1},\dots,\mu_{j_k}],$ $J=\{j_1,\dots,j_k\}$ (see \cite[Example 8.4.6]{ToricTopology}.) If $|J|\neq 2$ then $\mu_J$ does not belong to the ``diagonal'' subalgebra $\kK^!.$
\end{rmk}

\subsection{The K\"unneth formula for the Pontryagin algebras}
\label{subsection:free_module}
There is a split principal fibration of H-spaces,
$$\OZK\overset i \longrightarrow  \ODJ\overset p \longrightarrow \mathbb{T}^m$$
(see \cite[Sect. 8.4]{ToricTopology}),
and hence an exact sequence of Hopf algebras,
$$1\to H_*(\OZK;\k)\overset {i_*}\longrightarrow H_*(\ODJ;\k)\overset{p_*}\longrightarrow \Lambda[u_1,\dots,u_m]\to 0.$$
In particular, $i_*$ and $p_*$ are maps of $\k$-algebras. The map
$p:\ODJ\to\mathbb{T}^m\simeq\Omega(\CC\mathrm{P}^\infty)^{m}$
is induced by the map of simplicial complexes $\K\hookrightarrow \Delta^{m-1};$ therefore, it sends the element
$$u_i\in\Ext_{\kK}(\k,\k)\cong H_*(\ODJ;\k)$$
to the element
$$u_i\in\Ext_{\k[m]}(\k,\k)=H_*(\mathbb{T}^m;\k)\cong\Lambda[u_1,\dots,u_m].$$

Let $I$ be a subset of $[m].$ We order its elements: $I=\{i_1,\dots,i_k\},$ $i_1<\dots<i_k,$ and define
$$u_I := u_{i_1}\wedge\dots\wedge u_{i_k}\in \Lambda[u_1,\dots,u_m],\quad \widehat{u}_I:= u_{i_1}\cdot\dotso\cdot u_{i_k}\in H_*(\ODJ;\k).$$
Clearly,
$\{u_I\}_{I\subset[m]}$
is an additive basis of
$\Lambda[u_1,\dots,u_m].$
The morphism $p_*$ has a $\k$-linear section
$\sigma:\Lambda[u_1,\dots,u_m]\to H_*(\ODJ;\k),$
defined by $\sigma(u_I):=\widehat{u}_I.$
However, this map is not multiplicative unless $\sk_1\K=\sk_1\Delta^{m-1}.$

The morphism $i_*$ makes $H_*(\ODJ;\k)$ into a left $H_*(\OZK;\k)$-module: $x\cdot y := i_*(x)y.$ Consider the $\k$-linear map
$$T:H_*(\OZK;\k)\otimes \Lambda[u_1,\dots,u_m]\to H_*(\ODJ;\k),\quad T(x\otimes u_I):=i_*(x)\widehat{u}_I.$$

\begin{prp}\label{prp:left_module_isomorphism} For any $\K,$ the map $T$ is an isomorphism of left $H_*(\OZK;\k)$-modules. Moreover, the following diagram of $\k$-modules is commutative:
\begin{equation}
\label{eqn:T_commutative}
\xymatrix{
H_*(\OZK;\k)\otimes \Lambda[u_1,\dots,u_m]\ar[d]_{T}^{\cong}\ar[r]^-{\varepsilon\otimes\id}& \k\otimes\Lambda[u_1,\dots,u_m]\ar[d]_-{\lambda\otimes u_I\mapsto \lambda u_I}^-{\cong}\\
H_*(\ODJ;\k)\ar[r]^-{p_*}& \Lambda[u_1,\dots,u_m]
}
\end{equation}

\end{prp}
\begin{proof} If $x_1,x_2\in H_*(\OZK;\k)$ and $I\subset[m],$ then
\begin{align*}
T(x_1\cdot (x_2\otimes u_I))&= T(x_1x_2\otimes u_I) = i_*(x_1x_2)\widehat{u}_I = i_*(x_1)i_*(x_2)\widehat{u}_I\\
&=i_*(x_1)T(x_2\otimes u_I)=x_1\cdot T(x_2\otimes u_I).
\end{align*}
Hence $T$ is a $H_*(\OZK;\k)$-linear map.

The algebra $\Lambda[u_1,\dots,u_m]$ is a free $\k$-module. The fibration $p$ is trivial, and $\sigma$ is a section of $p_*.$ Hence $T$ is an isomorphism of $\k$-modules by the K\"unneth formula. Since $T$ is $H_*(\OZK;\k)$-linear, it is an isomorphism of $H_*(\OZK;\k)$-modules.

Finally, we check that the diagram \eqref{eqn:T_commutative}
 is commutative:
$$\xymatrix{
x\otimes u_I\ar@{|->}[r]\ar@{|->}[d] & \varepsilon(x)\otimes u_I\ar@{|->}[d]\\
i_*(x)\widehat{u}_I\ar@{|->}[r] & p_*(i_*(x))u_I
}$$
The map $p\circ i$ is trivial up to homotopy, hence $p_*\circ i_*=\varepsilon.$\qedhere
\end{proof}

\section{Relations in the Pontryagin algebras in the flag case}
\label{section:relations_algebras}

In this section we consider several chain complexes of free left $H_*(\ODJ;\k)$- and $H_*(\OZK;\k)$-modules. These algebras are $(\ZZ\times\Zm)$-graded; hence the complexes are $(\ZZ\times\ZZ\times\Zm)$-graded:\\$\deg H_{-i,2\alpha}(\ODJ;\k)=(0,-i,2\alpha),$ and the differentials usually have degree $(-1, 0, 0).$

\subsection{The minimal resolution of the $H_*(\ODJ;\k)$-module $\k$}
Let $0\neq \chi_\alpha\in\kKc,$ so $\supp\alpha\in\K.$ Note that $\supp (\alpha-e_j)\in\K$ for every $j\in \supp\alpha.$ Hence $\chi_{\alpha-e_j}\in\kKc$ is well defined. We let $$\deg\chi_\alpha:=(|\alpha|,-|\alpha|,2\alpha)\in\ZZ\times\ZZ\times\Zm$$ in this section.

\begin{prp} Let $\K$ be a flag simplicial complex. Then the left $H_*(\ODJ;\k)$-module $\k$ has a minimal resolution
\begin{multline}
\label{eqn:odj_resolution}
\dots\overset{d}\longrightarrow
H_*(\ODJ;\k)\otimes \kKc_{(2)}\overset{d}\longrightarrow
H_*(\ODJ;\k)\otimes \kKc_{(1)}\\
\overset{d}\longrightarrow
H_*(\ODJ;\k)\otimes \kKc_{(0)}\overset\varepsilon\longrightarrow
\k\to 0
\end{multline}
with the following differential of degree $(-1,0,0):$
\begin{equation}
\label{eqn:froberg_differential}
d: \kKc_{(t)}\to H_*(\ODJ;\k)\otimes \kKc_{(t-1)},\quad \chi_\alpha\mapsto \sum_{j\in \supp\alpha}u_j\otimes \chi_{\alpha-e_j}.
\end{equation}
\end{prp}
\begin{proof}
By Theorem~\ref{thm:hodj_description}, $H_*(\ODJ;\k)\cong\kK^!$ in the flag case. We use the Fr\"oberg resolution \cite{froberg} once again for the algebra $\kK^!\otimes\kK.$ This time the variables $u_i$ and $v_i$ correspond to $X_i$ and $Y_i.$ The definition of $\mathrm{Index}(m)$ (see \cite[p. 32]{froberg}) is symmetric; hence there is a resolution $(\kK^!\otimes\kK,~d_1+\dots+d_m)$ for the left $\kK^!$-module $\k.$ 

Now we describe the Fr\"oberg differentials $d_1,\dots,d_m$ explicitly. Let $x=u^{(\mu)}\otimes v^{(\nu)}\in \kK^!\otimes\kK$ be a monomial, that is, $x=u_{i_1}\dots u_{i_k}\otimes v_{j_1}\dots v_{j_l}$ for some $i_1,\dots,i_k,j_1,\dots,j_l\in [m]$ (this representation is usually not unique). The element $d_j(x)$ is defined as follows:
\begin{itemize}
\item If $x$ cannot be written as $c\cdot u^{(\mu)}\otimes v_jv^{(\nu')},$ where $c\in\k,$ then $d_j(x):=0;$
\item If $x=c\cdot u^{(\mu)}\otimes v_j v^{(\nu')},$ then $d_j(x):=c\cdot u^{(\mu)}u_i\otimes v^{(\nu')}.$
\end{itemize}
The algebra $\kK$ is commutative, and every monomial $v^{(\nu)}\neq 0$ can be written uniquely as $v^\alpha,$ $\supp\alpha\in\K.$ Therefore,
\begin{align*}
d_j(u^{(\mu)}\otimes v^{\alpha})&=
\begin{cases}0,&j\notin \supp\alpha;\\ u^{(\mu)}u_j\otimes v^{\alpha-e_j},&j\in \supp\alpha;\end{cases}\\
d(u^{(\mu)}\otimes v^\alpha)&=
\sum_{j=1}^m d_j(u^{(\mu)}\otimes v^\alpha)=\sum_{j\in \supp\alpha}u^{(\mu)}u_j\otimes v^{\alpha-e_j}.
\end{align*}
Replacing $v^\alpha\in\kK_{(|\alpha|)}$ with $\chi_\alpha\in\kKc_{(|\alpha|)},$ we obtain the differential \eqref{eqn:froberg_differential}.
\end{proof}

\subsection{The free resolution of the $H_*(\OZK;\k)$-module $\k$}

By Proposition \ref{prp:left_module_isomorphism}, there is an isomorphism of $(\ZZ\times\Zm)$-graded left $H_*(\OZK;\k)$-modules:
$$T:H_*(\OZK;\k)\otimes \Lm\to H_*(\ODJ;\k),\quad a\otimes u_I\mapsto i_*(a)\widehat{u}_I.$$

Consider the $\k$-modules $M_t:=\Lm\otimes \kKc_{(t)}.$ They are $(\ZZ\times\ZZ\times\Zm)$-graded: $$\deg u_I\otimes\chi_\alpha=\deg u_I+\deg\chi_\alpha=(0+|\alpha|,-|I|-|\alpha|,2I+2\alpha)=(t,-|I|-t,2I+2\alpha).$$

\begin{thm}\label{thm:ozk_resolution} Let $\K$ be a flag simplicial complex. Then the left $H_*(\OZK;\k)$-module $\k$ has a free resolution
\begin{equation}
    \label{eqn:ozk_resolution}
    \dots\to
    H_*(\OZK;\k)\otimes M_2\overset{\widehat{d}}\longrightarrow
    H_*(\OZK;\k)\otimes M_1\overset{\widehat{d}}\longrightarrow
    H_*(\OZK;\k)\otimes M_0\overset{\varepsilon}\longrightarrow
    \k\to 0
\end{equation}
with the following differential of degree $(-1,0,0):$
\begin{equation}
    \label{eqn:ozk_differential}
    \widehat{d}:M_t\to \Big(H_*(\OZK;\k)\otimes \underbrace{\Lm\Big)\otimes \kKc_{(t-1)}}_{=M_{t-1}},\quad u_I\otimes \chi_\alpha\mapsto \sum_{j\in \supp\alpha} T^{-1}(\widehat{u}_Iu_j)\otimes \chi_{\alpha-e_j}.
\end{equation}
\end{thm}

\begin{proof}
The chain complex \eqref{eqn:odj_resolution} is a free resolution of the left ${H_*(\ODJ;\k)}$-module $\k.$ The algebra ${H_*(\ODJ;\k)}$ is itself a free left ${H_*(\OZK;\k)}$-module. It follows that \eqref{eqn:odj_resolution} is a free resolution of the left $H_*(\OZK;\k)$-module~$\k.$

Now consider the following sequence of isomorphisms of left $H_*(\OZK;\k)$-modules:
$$\xymatrix{
H_*(\OZK;\k)\otimes M_t\ar@{=}[r] &
H_*(\OZK;\k)\otimes \Lm\otimes \kKc_{(t)}\ar[r]^-{T\otimes\id} &
H_*(\ODJ;\k)\otimes \kKc_{(t)}.
}$$
We will prove that it defines an isomorphism between the chain complexes \eqref{eqn:ozk_resolution} and \eqref{eqn:odj_resolution}:
$$\xymatrix{(H_*(\OZK;\k)\otimes M,\widehat{d})\ar[r]^-{T\otimes\id}&(H_*(\ODJ;\k)\otimes \kKc,d).}$$
It is sufficient to show that the diagram
$$\xymatrix{
H_*(\OZK;\k)\otimes M_t\ar[r]^-{\widehat{d}}\ar[d]_-{T\otimes\id}^-\cong & H_*(\OZK;\k)\otimes M_{t-1}\ar[d]_-{T\otimes\id}^-\cong \\
H_*(\ODJ;\k)\otimes \kKc_{(t)}\ar[r]^-{d} & H_*(\ODJ;\k)\otimes \kKc_{(t-1)}
}$$ is commutative.
For $x\in H_*(\OZK;\k)$ and $u_I\otimes \chi_\alpha\in \Lm\otimes\kKc_{(t)}= M_t,$ we have
$$\xymatrix{
x\otimes (u_I\otimes \chi_\alpha)\ar@{|->}[r]\ar@{|->}[d] & \sum_{j\in \supp\alpha} xT^{-1}(\widehat{u}_Iu_j)\otimes \chi_{\alpha-e_j}\\
i_*(x)\widehat{u}_I\otimes \chi_\alpha\ar@{|->}[r] & \sum_{j\in \supp\alpha} i_*(x)\widehat{u}_Iu_j\otimes \chi_{\alpha-e_j}.
}$$
But $T(xT^{-1}(\widehat{u}_Iu_j))=i_*(x) T(T^{-1}(\widehat{u}_Iu_j))=i_*(x)\widehat{u}_Iu_j,$ as $T$ is $H_*(\OZK;\k)$-linear and $H_*(\OZK;\k)$ acts on $H_*(\ODJ;\k)$ via $i_*.$
\end{proof}

\begin{rmk}
The resolution~\eqref{eqn:ozk_resolution} can be used to obtain a presentation of $H_*(\OZK;\k)$. However, an explicit description of $T^{-1}$ requires involved commutator calculus. Also, to obtain a presentation using assertion 2 of  Proposition~\ref{prp:homological}, one needs a resolution of the form $\dots\to H_*(\OZK;\k)\to \k\to 0$ instead of $$\dots\to H_*(\OZK;\k)\otimes\Lm\to\k\to 0.$$
\end{rmk}

\subsection{Proof of Theorem~\ref{thm:tor_ozk_answer} and a corollary to it}
\begin{prp}\label{prp:ozk_tor_calculation} Let $\K$  be a flag complex. Then
$\Tor^{H_*(\OZK;\k)}(\k,\k)\simeq H(M,\overline{d}),$ where $\overline{d}$ is defined as
$$\overline{d}: M_t\to M_{t-1},\quad u_I\otimes \chi_\alpha\mapsto \sum_{j\in \supp\alpha} (u_{I}\wedge u_j)\otimes \chi_{\alpha-e_j}.$$
\end{prp}
\begin{proof}
We apply the functor $\k\otimes_{H_*(\OZK;\k)} (-)$  to the resolution \eqref{eqn:ozk_resolution} and obtain a differential $M_t\to M_{t-1},$
$$u_I\otimes \chi_\alpha\mapsto ((\varepsilon\otimes\id)\otimes\id)\left(\sum_{j\in \supp\alpha} T^{-1}(\widehat{u}_Iu_j)\otimes \chi_{\alpha-e_j}\right)=\sum_{j\in \supp\alpha}(\varepsilon\otimes\id)(T^{-1}(\widehat{u}_Iu_j))\otimes \chi_{\alpha-e_j}.$$ The homology of this complex is isomorphic to $\Tor^{H_*(\OZK;\k)}(\k,\k)$ by definition.

It follows from \eqref{eqn:T_commutative} that
$(\varepsilon\otimes\id)(T^{-1}(\widehat{u}_Iu_j))=p_*(\widehat{u}_Iu_j)=u_I\wedge u_j.$
Hence
\[\sum_{j\in \supp\alpha}(\varepsilon\otimes\id)(T^{-1}(\widehat{u}_Iu_j))\otimes \chi_{\alpha-e_j}=\overline{d}(u_I\otimes \chi_\alpha).\qedhere
\]
\end{proof}

\begin{proof}[Proof of Theorem \ref{thm:tor_ozk_answer}]
We first consider the \emph{Koszul algebra} of the Stanley-Reisner ring. It is the following $(\ZZ\times\ZZ\times\ZZ^m)$-graded algebra supplied with a differential of degree $(1,0,0):$
$$(\Lambda[u_1,\dots,u_m]\otimes\kK,d),~\deg u_i=(0,-1,2e_i),~\deg v_i=(1,-1,2e_i);~du_i=v_i,~dv_i=0.$$
The differential is extended on the additive basis by the Leibniz rule: $$d=\sum_{i=1}^m \frac{\partial}{\partial u_i}\otimes v_i.$$

Now consider the dual coalgebra $(\Lambda[u_1,\dots,u_m]\otimes\kKc,\delta)$  with $$\delta = \sum_{i=1}^m u_i\otimes\frac{\partial}{\partial\chi_i}.$$
This differential coincides with $\overline{d}.$ Therefore, $\Tor^{H_*(\OZK;\k)}(\k,\k)$ is the homology of the $\k$-dual complex to the Koszul algebra.

Recall the calculation of the cohomology of the Koszul algebra. By \cite[Lemma 3.2.6]{ToricTopology}, factoring out the homogeneous acyclic ideal $(u_iv_i,v_i^2)$ gives a quasi-isomorphism $$(\Lm\otimes\kK,d)\to (\Lm\otimes\kK,d)/(u_iv_i=v_i^2=0)=:R(\K).$$
As shown in the proof of Theorem 3.2.9 in \cite{ToricTopology} there is an isomorphism of chain complexes
$$\widetilde{C}^*(\K_J;\k)\overset\cong\longrightarrow R^{*+1,-|J|,2J}(\K),~I\mapsto \pm u_{J\setminus I}\otimes v^I,$$
for every $J\subset[m].$ Moreover, the direct sum is an isomorphism
$$\bigoplus_{J\subset[m]}\widetilde{C}^*(\K_J;\k)\overset\cong\longrightarrow R(\K).$$
By passing to $\k$-dual complexes and then to homology, we obtain the isomorphisms $$\Tor_{n+1}^{H_*(\OZK;\k)}(\k,\k)_{-|J|,2J}\cong H_{n}(R^{*+1,-|J|,2J}(\K)^\#)\overset\cong\longrightarrow\H_n(\K_J;\k).$$ Their direct sum gives the required isomorphism
\[
\Tor^{H_*(\OZK;\k)}_{n+1}(\k,\k)\overset\cong\longrightarrow \bigoplus_{J\subset[m]}\H_n(\K_J;\k).
\qedhere
\]
\end{proof}

To simplify the notation, denote by $\dim M$ the minimal number of generators of a $\k$-module $M.$

\begin{crl}
\label{crl:number_of_gen_and_rel}
Let $\K$ be a flag simplicial complex and $\k$ be a commutative ring with unit. Then
\begin{enumerate}
\item[(a)] if $\k$ is a PID, any presentation of $H_*(\OZK;\k)$ contains at least ${\sum_{J\subset[m]}\dim\widetilde{H}_0(\K_J;\k)}$ generators and ${\sum_{J\subset[m]}\dim\widetilde{H}_1(\K_J;\k)}$ relations;
\item[(b)] if $\k=\FF$ is a field, the minimal presentation of the algebra $H_*(\OZK;\FF)$ contains exactly ${\sum_{J\subset[m]}\dim\widetilde{H}_0(\K_J;\FF)}$ generators and ${\sum_{J\subset[m]}\dim\widetilde{H}_1(\K_J;\FF)}$ relations.
\end{enumerate}
\end{crl}
\begin{proof}
Assertion (a) follows from Theorem~\ref{thm:tor_ozk_answer} and Proposition \ref{prp:homological_for_rings}.\\
Assertion (b) follows from Theorem~\ref{thm:tor_ozk_answer} and assertion 2 of Proposition \ref{prp:homological}.
\end{proof}

\begin{rmk}
The algebra $H_*(\OZK;\FF)$ is $(\ZZ\times\Zm)$-graded. This implies that the $\FF$-modules $\Tor_i^{H_*(\OZK;\FF)}(\FF,\FF)$ are $(\ZZ\times\Zm)$-graded. Hence, for every $J\subset[m],$ the corresponding $\dim\H_1(\K_J;\FF)$ relations in $H_*(\OZK;\FF)$ have degree $(-|J|,2J)\in\ZZ\times\Zm.$ Similarly, if $\k$ is a PID, then there are at least $\dim\H_1(\K_J;\k)$ relations of this degree in any presentation of $H_*(\OZK;\k).$

Note that the number of relations depends on $\k,$ since the number $\dim\H_1(\K_J;\k)$ does. For example, let $\K$ be a flag triangulation of $\RR\mathrm{P}^2.$ Then, in degree $(-m,2[m]),$ there are
\begin{itemize}
\item at least one relation in $H_*(\OZK;\ZZ);$
\item exactly one relation in $H_*(\OZK;\ZZ/2);$
\item no relations in $H_*(\OZK;\FF)$ if $\chr\FF\neq 2.$
\end{itemize}
It is natural to expect the following: if $\K$ is a flag triangulation of $\RR\mathrm{P}^2,$ then there is a relation $r=0$ in $H_*(\OZK;\k)$ of degree $(-m,2[m]),$ such that $2r=0$ follows from relations of smaller degree.
\end{rmk}

\begin{rmk}
\label{rmk:work_of_cai}
Recently, Li Cai \cite{lie_cai} obtained a method that produces a relation of degree $(-|J|, 2J)$ in $H_*(\OZK;\k),$ given a subset $J\subset[m]$ and a \emph{free simplicial loop} in $\K_J.$ We expect that
\begin{itemize}
\item the relation depends only on the homotopy class of the simplicial loop;
\item taking the generating loops for the fundamental groups of connected components of full subcomplexes, one obtains a sufficient set of relations.  
\end{itemize}
This expectation is based on the following evidence. Over a field, the ``module of minimal relations'' in $H_*(\OZK;\FF)$ is isomorphic to $\bigoplus_{J\subset[m]}\H_1(\K_J;\FF).$  But the first homology of any $\K_J$ is generated by the Hurewicz images of generating loops for fundamental groups of connected components.

Another argument is the analogy between the presentations of $H_*(\OZK;\k)$ and $\pi_1(\RK)$ (see \cite{pv, onerelator}). As shown by Li Cai, the analogous set of relations in $\pi_1(\RK)$ is sufficient by the van Kampen theorem.
\end{rmk}

\subsection{Multigraded Poincar\'e series of Pontryagin algebras and Lie superalgebras}
\label{subsection:pbw_use}

\begin{thm}
\label{thm:poincare_series}
Let $\K$ be a flag simplicial complex and $\FF$ be a field. Then
\begin{equation}
\label{eqn:poincare_series}
    \frac{1}{F(H_*(\OZK;\FF);t,\lambda)}=
    -\sum_{J\subset[m]}\widetilde{\chi}(\K_J)t^{-|J|}\lambda^{2J}.
\end{equation}
\end{thm}
\begin{proof}
It follows from Theorem~\ref{thm:tor_ozk_answer} that
$$F\left(\Tor_n^{H_*(\OZK;\FF)}(\FF,\FF);t,\lambda\right)=
\sum_{J\subset[m]}\dim_\FF \H_{n-1}(\K_J;\FF)t^{-|J|}\lambda^{2J}.$$

By assertion 3 of Proposition \ref{prp:homological},
\begin{align*}
\frac{1}{F(H_*(\OZK;\FF);t,\lambda)}&=\sum_{n=0}^\infty (-1)^n F\left(\Tor_n^{H_*(\OZK;\FF)}(\FF,\FF);t,\lambda\right)
=\\&
=\sum_{J\subset [m]}\sum_{n=0}^\infty (-1)^n \dim_\FF \H_{n-1}(\K_J)t^{-|J|}\lambda^{2J}=-\sum_{J\subset[m]}\widetilde{\chi}(\K_J)t^{-|J|}\lambda^{2J}.\qedhere
\end{align*}
\end{proof}

\begin{rmk}
Let $n=\dim\K.$ The \emph{$h$-numbers} of $\K$ are defined by the identity $$\sum_{i=0}^n h_i s^{n-i}=\sum_{I\in\K} (s-1)^{n-|I|}.$$
The following formula for the $\ZZ$-graded Poincar\'e series of $H_*(\OZK;\FF)$  follows from the results of Panov and Ray \cite{pr} (see \cite[Proposition 8.5.4]{ToricTopology}):
$$\frac{1}{F(H_*(\OZK;\FF);t)}=(1+t)^{m-n}\sum_{i=0}^n h_i\cdot(-t)^i.$$
(The $\ZZ$-graded Poincar\'e series $F(V;t)$ is obtained from the multigraded series $F(V;t,\lambda)$ by the substitution $t=\lambda_1=\dots=\lambda_m.$)
We show sketchily that this formula agrees with \eqref{eqn:poincare_series}.  We need to check that
\begin{equation}
\label{eqn:check_poincare}
(1+t)^{m-n}\sum_{i=0}^n h_i\cdot(-t)^i=-\sum_{J\subset[m]}\widetilde{\chi}(\K_J)t^{|J|}.
\end{equation}

First, we substitute $s=-1/t$ in the definition of $h$-numbers and obtain
$$
(1+t)^{-n}\sum_{i=0}^n h_i\cdot(-t)^i 
=
\sum_{I\in\K}(-t)^{|I|}(1+t)^{-|I|}.
$$
On the other hand, $-\widetilde{\chi}(\K_J)=\sum_{I\in\K_J}(-1)^{|I|}$ via simplicial homology. Hence \eqref{eqn:check_poincare} is equivalent to the following identity, which is easy to verify by changing the order of summation:
\[
\sum_{J\subset[m]}\sum_{I\in\K_J}(-1)^{|I|}t^{|J|}=\sum_{I\in\K}(-1)^{|I|}\cdot t^{|I|}(1+t)^{m-|I|}.
\qed
\]
\end{rmk}
~\\

Now we recall that, by the Milnor-Moore theorem \cite{milnor_moore}, the algebra $H_*(\OZK;\QQ)$ is the universal enveloping of the graded Lie superalgebra $\pi_*(\OZK)\otimes\QQ.$ Hence $\pi_*(\OZK)\otimes\QQ$ is $(\ZZ\times\Zm)$-graded:
$$\pi_n(\OZK)\otimes\QQ=\bigoplus_{n=-i+2|\alpha|}(\pi_*(\OZK)\otimes\QQ)_{-i,2\alpha},\quad (\pi_*(\OZK)\otimes\QQ)_{-i,2\alpha}\subset H_{-i,2\alpha}(\OZK;\QQ).$$

\begin{rmk}
A natural question is whether $\pi_*(\OZK)$ admits a natural $\ZZ\times\Zm$-grading.

If $\K$ is a flag complex, it may be possible to introduce such grading using the homotopy decomposition of $\OZK$ from \cite[Corollary 7.3]{pt}. The following argument serves as a motivation: by \cite[Theorem 1.1]{pt}, $\OZK$ is a homotopy retract of $\Omega\Z_{\overline{\K}},$ where $\overline{\K}$ is the disjoint union of $m$ points. But $\Z_{\overline{\K}}$ is a wedge of spheres, so an appropriate multigrading on the homotopy groups may be provided by the Hilton-Milnor theorem (see the proof of \cite[Theorem 5.3(c)]{gptw}). However, as in Remark \ref{rmk:panov_theriault_argument}, we do not know if the homotopy section from \cite{pt} respects the multigrading.
\end{rmk}

Now we use Theorem \ref{thm:poincare_series} to calculate the multigraded Poincar\'e series of $\pi_*(\OZK)\otimes\QQ$ for flag $\K.$ Let $L$ be a connected $(\ZZ\times\Zm)$-graded Lie superalgebra and $U(L)$ be its universal enveloping algebra. It is assumed that the first grading is either always positive or always negative, so we can divide power series and take logarithms. Define the coefficients $l_{i,\alpha}$ and $u_{j,\beta}$ by the identities
$$F(L;t,\lambda)=\sum_{i,\alpha} l_{i,\alpha}t^i\lambda^\alpha,\quad \ln F(U(L);t,\lambda)=\sum_{j,\beta} u_{j,\beta}t^j\lambda^\beta.$$
Define also
$$F(L;-t,-\lambda)=\sum_{i,\alpha} \widetilde{l}_{i,\alpha}t^i\lambda^\alpha,\quad \ln F(U(L);-t,-\lambda)=\sum_{j,\beta} \widetilde{u}_{j,\beta}t^j\lambda^\beta.$$
It is clear that $\widetilde{l}_{i,\alpha}=(-1)^{i+|\alpha|}l_{i,\alpha}$ and $\widetilde{u}_{j,\beta}=(-1)^{j+|\beta|}u_{j,\beta}.$

\begin{lmm}
The Poincar\'e series have the following properties:
\begin{enumerate}
\item $F(V_1\oplus V_2;t,\lambda)=F(V_1;t,\lambda)+F(V_2;t,\lambda);$
\item $F(V_1\otimes V_2;t,\lambda)=F(V_1;t,\lambda)\cdot F(V_2;t,\lambda).$
\end{enumerate}
\end{lmm}
\begin{proof}
If $\{v_i\}_{i\in I}$ is a basis of a vector space $V$ and $\deg v_i=(a_i,b_i)$ then $F(V;t,\lambda)=\sum_{i\in I}t^{a_i}\lambda^{b_i}.$ Property (1) follows from the fact that the union of bases is a basis of the direct sum, and property (2), from the fact that the set of pairwise tensor products of basic vectors is a basis of the tensor product.
\end{proof}

\begin{thm}[\cite{milnor_moore}]\label{thm:pbw_crl}
$$F(U(L);t,\lambda)=\prod_{i,\alpha} (1-(-t)^i(-\lambda)^\alpha)^{(-1)^{i+|\alpha|+1}l_{i,\alpha}}.$$
\end{thm}

\begin{proof}
The Poincar\'e-Birkhoff-Witt theorem implies that if $\{x_{i,\alpha}\}$ is a homogeneous additive basis  of $L,$ then $U(L)$ is isomorphic (as a graded vector space) to the free graded-commutative algebra generated by $\{x_{i,\alpha}\}.$ It follows that every summand $t^i\lambda^{\alpha}$ in $F(L;t,\lambda)$ corresponds to a factor
$$
\left.
\begin{cases}
F(\Lambda[x_{i,\alpha}];t,\lambda)=1+t^i\lambda^\alpha,&i+|\alpha|\text{ odd}\\
F(\FF[x_{i,\alpha}];t,\lambda)=\frac{1}{1-t^i\lambda^\alpha},&i+|\alpha|\text{ even}
\end{cases}
\right\rbrace = (1-(-1)^{i+|\alpha|}t^i\lambda^{\alpha})^{(-1)^{i+|\alpha|+1}}$$ in $F(U(L);t,\lambda).$ 
\end{proof}
\begin{lmm}
\label{lmm:poincare_lie_explicit_formula}
$$\displaystyle{\widetilde{l}_{i,\alpha}=\sum_{k\mid i,\alpha}\widetilde{u}_{i/k,\alpha/k}\frac{\mu(k)}k,}$$ where $\sum_{k\mid i,\alpha}$ denotes summation over all positive integer $k$ dividing $i$ and that $\alpha_1,\dots,\alpha_m,$ and $\mu(n)$ is the M\"obius function.
\end{lmm}
\begin{proof}
By Theorem~\ref{thm:pbw_crl}, $F(U(L);-t,-\lambda)=\prod_{i,\alpha} (1-t^i\lambda^\alpha)^{-\widetilde{l}_{i,\alpha}}$ . Hence
$$\sum_{j,\beta} \widetilde{u}_{j,\beta}t^j\lambda^\beta=-\sum_{i,\alpha} \widetilde{l}_{i,\alpha}\ln(1-t^i\lambda^\alpha)=\sum_{i,\alpha} \widetilde{l}_{i,\alpha}\sum_{k=1}^\infty\frac{t^{ik}\lambda^{k\alpha}}k=\sum_{j,\beta}\sum_{k\mid j,\beta}\widetilde{l}_{j/k,\beta/k}\frac{t^j\lambda^{\beta}}k.$$
Equivalently, $\widetilde{u}_{j,\beta} =\sum_{k\mid j,\beta} \widetilde{l}_{j/k,\beta/k}\cdot\frac1k.$
The claim now follows from the ($(\ZZ\times\Zm)$-graded) M\"obius inversion formula.\qedhere
\end{proof}

\begin{thm}
\label{thm:homotopy_ranks}
Let $\K$ be a flag simplicial complex. Consider the $\Zm$-graded power series
$$
    \sum_\beta w_{\beta}\lambda^\beta =
    -\ln\left(-\sum_{J\subset[m]} \widetilde{\chi}(\K_J)(-\lambda)^{J}\right).
$$
Then $$
    \dim (\pi_*(\OZK)\otimes\QQ)_{-|\alpha|,2\alpha}=
    (-1)^{|\alpha|}\sum_{k\mid \alpha}\frac{\mu(k)}k w_{\alpha/k}
$$
for every $\alpha\in\Zm.$ Otherwise, $\dim (\pi_*(\OZK)\otimes\QQ)_{i,\beta}=0.$
\end{thm}
\begin{proof}
Let $L=\pi_*(\OZK)\otimes\QQ,$ so $U(L)=H_*(\OZK;\QQ)$ by the Milnor-Moore theorem. %Then $U(L)$ is concentrated in degrees $\{(-i,2\alpha):~i\geq 0,~i=|\alpha|\}\subset\ZZ\times\Zm,$ so the same is true for 
Then Theorem \ref{thm:poincare_series} gives
$$\sum_{i,\beta}\widetilde{u}_{i,\beta}t^i\lambda^\beta=-\ln\left(-\sum_{J\subset[m]}\widetilde{\chi}(\K_J)(-t)^{-|J|}\lambda^{2J}\right)=
-\ln\left(-\sum_{J\subset[m]}\widetilde{\chi}(\K_J)(-\lambda^2/t)^J\right)=
\sum_\alpha w_\alpha t^{-|\alpha|}\lambda^{2\alpha}.$$
Thus $\widetilde{u}_{-|\alpha|,2\alpha}=w_\alpha,$ and $\widetilde{u}_{i,\beta}=0$ in other cases. By Lemma \ref{lmm:poincare_lie_explicit_formula},
$$\widetilde{l}_{-|\alpha|,2\alpha}=\sum_{k\mid -|\alpha|,2\alpha}\widetilde{u}_{-|\alpha|/k,2\alpha/k}\frac{\mu(k)}{k}=\sum_{k\mid \alpha}w_{\alpha/k}\frac{\mu(k)}k,$$
and $\widetilde{l}_{i,\beta}=0$ in other cases Finally, $\dim(\pi_*(\OZK)\otimes\QQ)_{-|\alpha|,2\alpha}=l_{-|\alpha|,2\alpha}=(-1)^{|\alpha|}\widetilde{l}_{-|\alpha|,2\alpha}.$
\end{proof}

\begin{rmk}The numbers
$$l_n:=\dim \pi_{n+1}(\ZK)\otimes\QQ=\dim\pi_{n}(\OZK)\otimes\QQ=\sum_{|\alpha|=n}l_{-|\alpha|,2\alpha}$$
for flag $\K$ were first described by Denham and Suciu \cite[Theorem 4.2.1]{denham_suciu}. They followed the same scheme: the $\ZZ$-graded variant of Theorem \ref{thm:pbw_crl} gives
$$F(H_*(\OZK;\QQ);t)=\prod_{r\geq 1} \frac{(1+t^{2r-1})^{l_{2r}}}{(1-t^{2r})^{l_{2r+1}}}.$$ The \emph{L\"ofwall formula} is then used to express $F(H_*(\OZK;\QQ);t)$ through the bigraded Poincar\'e series of $H^*(\ZK;\QQ).$ Sometimes this approach simplifies the computations, as in \cite[Example 7.2.1]{denham_suciu}, but in general our formula from Theorem \ref{thm:homotopy_ranks} is more explicit. It also gives the multigraded answer.
\end{rmk}

\begin{rmk} In \cite{ustinovskiy}, Ustinovskiy interpreted the \emph{non-negativity} of numbers $l_n=\dim(\pi_n(\OZK)\otimes\QQ)$  as certain polynomial inequalities on the $h$-numbers of $\K$ in the flag case (the $h$-numbers appear as coefficients in the Panov-Ray formula). 

In a similar way, using the multigrading, we can interpret the non-negativity of $l_{-|\alpha|,2\alpha}=\dim(\pi_*(\OZK)\otimes\QQ)_{-|\alpha|,2\alpha}$ as polynomial inequalities for the reduced Euler characteristics of full subcomplexes of $\K.$ These inequalities imply those by Ustinovsky, since $l_n=\sum_{|\alpha|=n}l_{-|\alpha|,2\alpha}.$ However, their combinatorial interpretation is subtler. In particular, in the multigraded case the logarithm of a power series cannot be expressed through the Newton polynomials, as in \cite[Lemma 1.4]{ustinovskiy}.
\end{rmk}

\begin{exm}
\label{exm:restriction_on_chi}
Consider the inequality $\dim (\pi_*(\OZK)\otimes\QQ)_{-|\alpha|,2\alpha}\geq 0$ for $\alpha=[m]=\sum_{i=1}^me_i\in\Zm.$ In this case $k|\alpha$ only if $k=1,$ so Theorem \ref{thm:homotopy_ranks} gives
$$\dim (\pi_*(\OZK)\otimes\QQ)_{-m,2[m]}=(-1)^{m}w_{[m]},\quad\sum_\beta w_\beta\lambda^\beta=-\ln\left(1-\sum_{J\subset[m],~J\neq\varnothing}(-1)^{|J|}\widetilde{\chi}(\K_J)\lambda^J\right).$$
Expanding the right side using the formula $-\ln(1-t)=\sum_{N\geq 1}t^N/N$ and taking the coefficient of $\lambda^{[m]},$ we obtain the following inequality: for any flag complex $\K,$

\begin{equation}
    \sum_{N\geq 1}\frac 1N\sum_{[m]=J_1\sqcup\dots\sqcup J_N}^{J_i\neq\varnothing}\,\prod_{i=1}^N\widetilde{\chi}(\K_{J_i})\geq 0.
\end{equation}

Here the subsets $J_1,\dots,J_N$ are ordered, so every partition $\{J_1,\dots,J_N\}$ of $[m]$ appears in the sum exactly $N!$ times.

More generally, if $\mathrm{gcd}(\alpha_1,\dots,\alpha_m)=1$ then similar calculations give the inequality
$$\dim(\pi_*(\OZK)\otimes\QQ)_{-|\alpha|,2\alpha}=\sum_{N\geq 1}\frac{1}N\sum_{\alpha=J_1+\dots+J_N}^{J_i\neq\varnothing}\prod_{i=1}^N\widetilde{\chi}(\K_{J_i})\geq 0.$$
\end{exm}

\section{LS-category of moment-angle complexes in the flag case}
\label{section:ls}
\begin{dfn}
Let $X$ be a topological space. The \emph{LS-category} $\cat(X)$ is the smallest integer $n,$ such that there is an open covering $X=U_0\cup\dots\cup U_n$ where every $U_i\hookrightarrow X$ is null-homotopic.
\end{dfn}
\begin{rmk}
In early sources, such as \cite{ginsburg}, the covering was denoted by $U_1\cup\dots\cup U_n.$ Hence the definition of $\cat(X)$ was shifted by $1.$ In other papers we refer to, the modern convention is used.
\end{rmk}
Lower and upper bounds on $\cat(\ZK)$ were obtained in \cite{beben_grbic}. We use the following result.
\begin{prp}[{\cite[Corollary 3.13]{beben_grbic}}]
\label{prp:beben_grbic_inequality} Let $\K$ be a simplicial complex without ghost vertices. Then $\cat(\ZK)\le \cat(\RK).$\qed
\end{prp}

\subsection{LS-category of real moment-angle complexes in the flag case}
\label{subsection:cat_rk}

For flag complexes, the calculation of $\cat(\RK)$ reduces to a group-theoretical problem which was solved by Dranishnikov \cite{dranishnikov}.

\begin{thm}[see \cite{pv}]
If $\K$ is flag, then $\RK = B(\RCK'),$ where $\RCK'$ is the commutator subgroup of the right-angled Coxeter group $\RCK.$\qed
\end{thm}

\begin{dfn} The \emph{cohomological dimension} $\cdim_\k X$ of a topological space $X$ over a ring $\k$ is the smallest integer $n$ such that $\H^i(X;\k)=0$ for all $i>n.$ Similarly, the \emph{homological dimension} $\hdim_\k X$ is the smallest $n$ such that $\H_i(X;\k)=0$ for all $i>n.$
\end{dfn}

\begin{dfn} The \emph{cohomological dimension} $\cd G$ of a discrete group $G$ is the smallest integer $n$ such that $H^i(G;M)=0$ for all $i>n$ and all $\ZZ[G]$-modules $M.$
\end{dfn}

To cover the degenerate cases, we set $\cdim_\k\pt:=-1$ and $\cd 1:=0.$ Note that $\cdim_\k\varnothing:=-1.$

\begin{thm}
\label{thm:eg} For any discrete group $G,$ one has $\cd G=\cat(BG).$ 
\end{thm}
\begin{proof}
By the Eilenberg-Ganea theorem~\cite{eilenberg_ganea}, the inequality $\cd G\neq\cat(BG)$ is possible only if $\cd G=1.$ In that case $G$ is free by the Stallings-Swan theorem~\cite{swan}. Therefore, $BG$ is a wedge of circles, so $\cat(BG)=1.$
\end{proof}

\begin{prp}[{\cite[Proposition VIII.2.2]{brown}}]
\label{prp:cd_vs_cdim} It holds that
$\cdim_\ZZ BG \leq \cd G\leq \dim BG.$\qed
\end{prp}
\begin{crl}
\label{crl:cat_rk_lower_bound}
If $\K$ is a flag complex, then $\cdim_\ZZ\RK\leq \cd\RCK'=\cat(\RK)\leq\dim\RK.$\qed
\end{crl}

\begin{prp}
\label{prp:cdim_rk}
If $\K\neq\Delta^{m-1},$ then
$\cdim_\ZZ\RK=1+\max_{J\subset[m]}\cdim_\ZZ\K_J.$
\end{prp}
\begin{proof}
By Theorem~\ref{thm:moment_angle_homology}, $H^n(\RK;\ZZ)=\bigoplus_{J\subset[m]}\widetilde{H}^{n-1}(\K_J;\ZZ).$ Since $\K\neq\Delta^{m-1},$ there is a missing face in $\K;$ hence $H^{s-1}(\RK;\ZZ)\neq 0,$ where $s\geq 2$ is the size of the missing face. Therefore, $\cdim_\ZZ\RK\geq 1,$ and if $n\geq 1$ then $H^n(\RK;\ZZ)=\H^n(\RK;\ZZ).$
\end{proof}
Recall the notion of the \emph{virtual cohomological dimension} of a discrete group $G:$ if $G_1\subset G$ is a subgroup of finite index, and $\cd G_1<\infty,$ then $\vcd G := \cd G_1.$ This number does not depend on the choice of $G_1$ by a result of Serre (see \cite[\S VIII.11]{brown}).
\begin{prp}
[cf. {\cite[p. 142]{dranishnikov}}]
\label{prp:dranishnikov_formula}
Let $\K$ be a flag simplicial complex. Then
$$\cd\RCK'=1+\max_{I\in\K}\cdim_\ZZ\lk_\K I.$$
\end{prp}
\begin{proof}
We have $\vcd\RCK = \cd\RCK',$ since $[\RCK : \RCK'] =  |\ZZ_2^m|<\infty$ and $\cd\RCK'\leq \dim\RK <\infty.$

In \cite{dranishnikov}, Dranishnikov proved the following formula:
$$\vcd\RCK=\max(\mathrm{lcd}(\K),1+\cdim_\ZZ\K),\quad\text{where }\mathrm{lcd}(\K):=1+\max_{I\in\K,I\neq\varnothing}\cdim_\ZZ\lk_\K I.$$
(In the notation of \cite{dranishnikov}, the empty set is not a simplex.) Since $\lk_\K \varnothing =\K,$ the claim follows.
\end{proof}

\begin{prp}
\label{prp:cat_rk_calculation}
For every flag $\K\neq\Delta^{m-1},$ we have $$\cd\RCK'=\cat(\RK)=\cdim_\ZZ\RK=1+\max_{J\subset[m]}\cdim_\ZZ\K_J.$$
\end{prp}
\begin{proof}
By Corollary \ref{crl:cat_rk_lower_bound}, Proposition \ref{prp:dranishnikov_formula}, Lemma \ref{lmm:lk_is_full_subcomplex} and Proposition \ref{prp:cdim_rk},
\[
\cdim_\ZZ\RK\leq\cat(\RK)=\cd\RCK'=1+\max_{I\in\K}\cdim_\ZZ\lk_\K I\leq 1+\max_{J\subset[m]}\cdim_\ZZ\K_J=\cdim_\ZZ\RK.
\qedhere
\]
\end{proof}
\begin{rmk} It follows that
$$\max_{I\in\K}\cdim_\ZZ\lk_\K I = \max_{J\subset[m]}\cdim_\ZZ\K_J$$ for flag $\K.$ In fact, this identity holds for any simplicial complex. It can be obtained by applying the \emph{combinatorial Alexander duality} (see \cite[Corollary 2.4.6]{ToricTopology}) to the results of Ayzenberg \cite[Proposition 3.2]{ayz}. He proved that the \emph{s-link-acyclic} simplicial complexes are exactly \emph{s-subcomplex-acyclic} complexes.
\end{rmk}

\subsection{Lower bounds on the LS-category of moment-angle complexes}
Let $\FF$ be a field and $X$ be a simply connected CW-complex. There exists a first quadrant homological spectral sequence, called the \emph{Milnor-Moore spectral sequence}, with the following properties:
$$E^2_{p,q}\cong\Tor^{H_*(\Omega X;\FF)}_p(\FF,\FF)_q\Rightarrow H_{p+q}(X;\FF).$$

It can be constructed as follows (see \cite[Sect. II.A]{toomer}). Let $(A,d)$ be a connected differential graded algebra over $\FF.$ Then the reduced bar construction $\oB(A):=\B(\FF,\FF)$ (see Subsection \ref{subsection:bar}) of $A$ is a bicomplex (there is a differential $\partial$ of the bar construction, and a differential $d$ which is extended on $\oB(A)$ from $(A,d)$ by the Leibniz rule.) Hence there is a spectral sequence with the second page
$$E^2=H(H(\oB(A),d),\partial)\cong H(\oB(H(A)),\partial)\cong \Tor^{H(A)}(\FF,\FF)$$
(the isomorphism $H(\oB(A),d)\cong\oB(H(A))$ follows from the K\"unneth theorem.)

The Milnor-Moore spectral sequence is obtained by taking $A=C_*(\Omega X;\FF).$ 
Other approaches use natural geometric filtrations on certain spaces that are weakly equivalent to $X:$
\begin{itemize}
    \item Milnor's construction of the classifying space $B_\infty G(X)$ of a topological group $G(X)\simeq\Omega X$ (see \cite{milnor} as well as \cite[Sect. 6]{buchstaber_limonchenko});
    \item Ganea's fibre-cofibre construction \cite{ganea};
    \item G.Whitehead's filtration on the space of maps $\Delta^\infty\to X$ (see \cite[\S1]{ginsburg}).
\end{itemize}
These constructions give isomorphic spectral sequences (see \cite[\S 3]{ginsburg}). Geometric interpretations show that the Milnor-Moore spectral sequence converges to $H_*(X;\FF).$

The LS-category of $X$ is bounded from below by the \emph{Toomer invariant} $e_\FF(X).$
\begin{thm}[{\cite[Theorem~2.2]{ginsburg}, see also \cite[Theorem II.B.2]{toomer}}]
\label{thm:ginsburg}
For the Milnor-Moore spectral sequence, let $e_\FF(X):=\max\{p:~E^\infty_{p,*}\neq 0\}.$ Then $\cat(X)\geq e_\FF(X).$\qed
\end{thm}
\begin{prp}
\label{prp:mm_collapse}
Let $\K$ be a flag simplicial complex and $\FF$ be a field. Then the Milnor-Moore spectral sequence for $\ZK$ over $\FF$ collapses on $E^2,$ and $e_\FF(\ZK)=1+\max_{J\subset[m]}\hdim_\FF\K_J.$
\end{prp}
\begin{proof}

Using Theorem~\ref{thm:tor_ozk_answer} and Theorem~\ref{thm:moment_angle_homology}, we calculate the total dimensions of $E^2_{*,*}$ and $E^\infty_{*,*}:$
\begin{align*}
\sum_{p,q}\dim E_{p,q}^2&=\sum_i\dim \Tor_i^{H_*(\OZK;\FF)}(\FF,\FF)=\sum_{i}\sum_{J\subset[m]}\dim\widetilde{H}_{i-1}(\K_J;\FF),\\
\sum_{p,q}\dim E_{p,q}^\infty &=\sum_i \dim H_i(\ZK;\FF)=\sum_i\sum_{J\subset[m]}\dim\widetilde{H}_{i-|J|-1}(\K_J;\FF).
\end{align*}
They coincide; hence the differentials are trivial. The formula for the Toomer invariant follows from Theorem \ref{thm:tor_ozk_answer}: $E_{p,*}^\infty\cong E^2_{p,*}\cong \Tor^{H_*(\OZK;\FF)}_p(\FF,\FF)=\bigoplus_{J\subset[m]}\H_{p-1}(\K_J;\FF).$
\end{proof}

\begin{lmm}
\label{lmm:cdim_and_hdim}
Let $X$ be a topological space with finitely generated homology groups. Then $\cdim_\ZZ X=\max_\FF\hdim_\FF X,$ where the maximum is over all fields $\FF.$
\end{lmm}
\begin{proof}
Let $\cdim_\ZZ X=n$ and let $\FF$ be a field. By the universal coefficient theorem,
$\H_k(X;\FF)$ can be expressed through $\H^k(X;\ZZ)$ and $\H^{k+1}(X;\ZZ);$
therefore, $\widetilde{H}_k(X;\FF)=0$ for $k>n.$ Hence $\hdim_\FF X\leq n$ for all $\FF.$

It remains to show that $\H_n(X;\FF)\neq 0$ for some $\FF.$ Since $\H^n(X;\ZZ)\neq 0,$ there are two possibilities:
\begin{itemize}
\item $\H^n(X;\ZZ)$ has nontrivial free part;
\item $\H^n(X;\ZZ)$ has nontrivial $p$-torsion for some $p.$
\end{itemize}
In the first case, $\H^n(X;\QQ)\neq 0,$ and it follows that $\H_n(X;\QQ)\neq 0.$ In the second, $\H_{n-1}(X;\ZZ)$ has nontrivial $p$-torsion, and hence $\H_{n-1}(X;\ZZ/p)$ and $\H_n(X;\ZZ/p)$ have nontrivial $p$-torsion. So $\H_n(X;\ZZ/p)\neq 0.$
\end{proof}

\begin{prp}
\label{prp:cat_zk_lower_bound_flag}
Let $\K$ be a flag simplicial complex. Then $\cat(\ZK)\ge 1+\max_{J\subset[m]}\cdim_\ZZ\K_J.$
\end{prp}
\begin{proof}
$\cat(\ZK)\geq\max_\FF e_\FF(\ZK)=1+\max_{J\subset[m]}\max_{\FF}\hdim_\FF\K_J=1+\max_{J\subset[m]}\cdim_\ZZ\K_J.$
\end{proof}

Now we can prove Theorem \ref{thm:cat_final}.
\begin{proof}[Proof of Theorem \ref{thm:cat_final}]
Propositions \ref{prp:beben_grbic_inequality},
\ref{prp:cat_rk_calculation},
\ref{prp:cat_zk_lower_bound_flag} give a chain of cyclic inequalities between $\cat(\ZK),$ $\cat(\RK),$ $\cd\RCK',$ and $1+\max_{J\subset[m]}\cdim_\ZZ\K_J.$
\end{proof}
\begin{crl}
\label{crl:cat_manifold}
If $\K$ is a flag triangulation of a $d$-manifold, then $\cat(\ZK)=\cat(\RK)=d+1.$
\end{crl}
\begin{proof}
We need to prove that $\max_{J\subset[m]}\cdim_\ZZ\K_J=d.$ We have
\[
d=\hdim_{\ZZ/2}\K\leq\cdim_\ZZ\K\leq\max_{J\subset[m]}\cdim_\ZZ\K_J\leq \dim\K=d.\qedhere
\]
\end{proof}
\begin{rmk}
The LS-category of $\ZK$ for arbitrary triangulations of 1-and 2-manifolds was calculated by Beben and Grbi\'c \cite[Propositions 4.7, 4.12]{beben_grbic}. In the non-flag case, $\cat(\ZK)$ can be smaller than $d+1.$
\end{rmk}

Denote the set of missing faces of $\K$ by $\MF(\K).$ By definition, $\MF(\K)=\{I\subset [m]:I\notin\K,~I\setminus\{i\}\in\K~\forall i\in I\}.$ To give a lower bound in the non-flag case, we use an inequality noted by Beben and Grbi\'c:

\begin{prp}[{see \cite[Lemma 3.11]{beben_grbic}}]
\label{prp:beben_grbic_filtration_bounds}
Let $\L_0\subset\dots\subset\L_s$ be a filtration of simplicial complexes such that $\L_{p+1}\setminus\L_p\subset\MF(\L_p)$ for all $p=0,\dots,s-1.$
Then $\cat(\Z_{\L_s})\leq \cat(\Z_{\L_0})+s.$
\end{prp}
\begin{proof}
The argument is the same as in \cite[Lemma 3.10 and Lemma 3.11]{beben_grbic}. The key step is to verify the decomposition
\begin{equation}
\label{eqn:set_theoretic}
\Z_{\L_{p+1}}\setminus\Z_{\L_p}=\bigsqcup_{I\in\L_{p+1}\setminus\L_p} \prod_{i\in I}(D^2\setminus S^1)\times\prod_{i\notin I}S^1.
\end{equation}
Let $x=(x_1,\dots,x_m)\in\Z_{\L_{p+1}}\setminus\Z_{\L_p}.$ There exists a simplex $I\in\L_{p+1}\setminus\L_p$ such that $x\in\prod_{i\in I}D^2\times\prod_{i\notin I}S^1.$ By assumption, for any $j\in I$ one has $I\setminus\{j\}\in \L_p.$ If there is an index $j\in I$ such that $x_j\in S^1,$ then $$x\in \prod_{i\in I\setminus\{j\}}D^2\times\prod_{i\notin I\setminus\{j\}}S^1\subset\Z_{\L_p},$$
a contradiction. Hence $x_j\in D^2\setminus S^1$ for every $j\in I.$ This argument proves that 
$$\Z_{\L_{p+1}}\setminus\Z_{\L_p}\subset\bigsqcup_{I\in\L_{p+1}\setminus\L_p} \prod_{i\in I}(D^2\setminus S^1)\times\prod_{i\notin I}S^1.$$
Conversely, let $x\in \prod_{i\in I}(D^2\setminus S^1)\times\prod_{i\notin I}S^1$ for some $I\in\L_{p+1}\setminus\L_p.$ Clearly, $$x\in \prod_{i\in I}D^2\times\prod_{i\notin I}S^1\subset \Z_{\L_{p+1}}.$$ If $x\in\Z_{\L_p},$ then there is a simplex $J\in\L_p$ such that $x\in\prod_{i\in J}D^2\times\prod_{i\notin J}S^1.$ It follows that $I\subset J,$ and this contradicts the assumptions $I\notin\L_p$ and $J\in\L_p.$ So $x\in\Z_{\L_{p+1}}\setminus\Z_{\L_p},$ and \eqref{eqn:set_theoretic} is verified.
\end{proof}

Now we define a number $\nu(\K)$ that measures how far $\K$ is from being flag. Recall that $\Kf$ is the unique flag simplicial complex with $\sk_1\Kf=\sk_1\K.$
\begin{dfn}
\label{dfn:flag_filtration}
Let $\K$ be a simplicial complex. Consider the following filtration of simplicial complexes:
$$\L_0:=\K,\quad\L_{p+1}:=\L_p\cup\{I\in \MF(\L_p):~|I|\geq 3\}$$
(on each step, we add to the complex $\L_p$ all its missing faces that are not edges). 

Clearly, the filtration stabilizes on a flag simplicial complex. Since for each $p$ we have $\sk_1\L_p=\sk_1\L_{p-1}=\dots=\sk_1\K,$ it stabilizes on $\Kf.$ Let $\nu(\K)$ be the minimal $n\geq 0$ such that $\L_n=\Kf.$% This number measures, how far $\K$ is from being flag.
\end{dfn}

The number $\nu(\K)$ and complexes $\L_p$ admit an explicit description.
\begin{prp}
\label{prp:nu_set_theoretic_description}
Let $\K$ be a simplicial complex. Then $\nu(\K)$ is the smallest $n\geq 0$ such that the following holds: $J\in\K$ whenever $J\subset I\in\Kf$ and $|I\setminus J|\geq n.$
\end{prp}
\begin{proof}
Consider the simplicial complexes
$$\L_p':=\{I\in\Kf\mid J\in\K\text{ for every }J\subset I\text{ such that }|I\setminus J|\geq p\}.$$
Clearly, $\L_p'\subset\L_{p+1}'.$
The condition that $J\in\K$ whenever $J\subset I\in\Kf$ and $|I\setminus J|\geq n$ is equivalent to $\L_n'=\Kf,$ so it is sufficient to show that $\L_p'=\L_p$ for all $p\geq 0.$ We use induction on $p.$

\emph{Base case:} $\L_0'=\{I\in\Kf\mid J\in\K\text{ for every }J\subset I\}=\K=\L_0.$

\emph{Inductive step:} let $\L_p'=\L_p.$ To prove that $\L_{p+1}'=\L_{p+1},$ we show that $\L_{p+1}\setminus\L_p\subset\L_{p+1}'$ and $\L_{p+1}'\setminus\L_p'\subset\L_{p+1}.$

\begin{enumerate}
\item Let $I\in\L_{p+1}\setminus\L_p.$ Then $I\in\Kf$ and $I\in\MF(\L_p).$ Suppose that $J\subset I$ and $|I\setminus J|\geq p+1.$ To prove that $I\in\L_{p+1}',$ it is sufficient to show that $J\in\K.$ 

Since $|I\setminus J|\geq p+1\geq 1,$ there exists $i\in I\setminus J.$ Clearly, $J\subset I\setminus\{i\}$ and $|(I\setminus\{i\})\setminus J|\geq p.$ Also $I\setminus\{i\}\in\L_p',$ since $I\in\MF(\L_p').$ Hence $J\in\K.$

\item Let $I\in\L_{p+1}'\setminus\L_{p}',$ so $I\in\Kf.$ To show that $I\in\L_{p+1},$ we prove that $I\in\MF(\L_p)$ and $|I|\geq 3.$
\begin{itemize}
\item Let $i\in I,$ and suppose that $J\subset I\setminus\{i\}$ and $|(I\setminus\{i\})\setminus J|\geq p.$ Clearly, $J\subset I$ and $|I\setminus J|\geq p+1.$ Since $I\in\L_{p+1}',$ it follows that $J\in\K.$ This argument shows that $I\setminus\{i\}\in\L_p'.$ Hence $I\in\MF(\L_p')=\MF(\L_p).$
\item Suppose $|I|\leq 2.$ Then $I\in\sk_1\L'_{p+1}.$ But $\L'_{p+1}\subset\Kf,$ so $I\in\sk_1\Kf=\sk_1\L_{\nu(\K)}=\dots=\sk_1\L_p,$ a contradiction with $I\notin\L_p.$\qedhere
\end{itemize}
\end{enumerate}
\end{proof}

\begin{crl}
\label{crl:nu_skeleton_bound}
Let $\K$ be a simplicial complex. Let $d=\dim\Kf,$ and suppose $\sk_i\K=\sk_i\Kf$ for some $i\leq d.$ Then $\nu(\K)\leq d-i.$
\end{crl}
\begin{proof}
By Proposition \ref{prp:nu_set_theoretic_description}, it is sufficient to prove that $J\in\K$ whenever $J\subset I\in\Kf$ and $|I\setminus J|\geq d-i.$ But if $I\in\Kf,$ then $|I|\leq d+1.$ Hence $|J|\leq i+1,$ so $J\in\sk_i\Kf=\sk_i\K\subset\K.$
\end{proof}
\begin{exm}
Let $\K$ be the $i$-skeleton of a flag complex $\Kf,$ $2\leq i\leq d=\dim\Kf.$ Then $\nu(\K)\leq d-i$ by Corollary \ref{crl:nu_skeleton_bound}.
On the other hand, choose a $d$-simplex $I\in\Kf$ and a subset $J\subset I,$ $|J|=i+2.$ Then $J\notin\K$ and $|I\setminus J|=(d+1)-(i+2)=d-i-1.$ It follows that $\nu(\K)=d-i.$ In particular, $\nu(\sk_i\Delta^{m-1})=m-i-1.$

However, the inequality $\nu(\K)\leq \min\{n:\sk_{d-n}\K=\sk_{d-n}\Kf\}$ can be strict. For example, if $\K=\partial\Delta^2\sqcup\partial\Delta^5$ then $\nu(\K)=1$ and $d=5,$ but $\sk_4\K\neq\sk_4\Kf.$
\end{exm}
\begin{proof}[Proof of Proposition \ref{prp:cat_zk_lower_bound}]
Clearly, the filtration $\K=\L_0\subset\dots\subset\L_{\nu(\K)}=\Kf$ from Definition \ref{dfn:flag_filtration} satisfies the conditions of Proposition \ref{prp:beben_grbic_filtration_bounds}. Hence $\cat(\Z_{\Kf})\leq\cat(\ZK)+\nu(\K).$ It remains to apply \ref{thm:cat_final}.
\end{proof}
\begin{crl}
\label{crl:cat_sk_manifold}
Let $\K$ be a flag triangulation of a $d$-manifold. Let $\L$ be a simplicial complex on $[m]$ such that $\sk_i\K\subset\L\subset\sk_j\K$ for some $i,j,$ $1\leq i\leq j\leq d.$ Then $i+1\leq\cat(\Z_\L)\leq j+1.$ In particular, $\cat(\Z_{\sk_i\K})=i+1.$
\end{crl}
\begin{proof}
By Corollary \ref{crl:cat_manifold}, we have $\cat(\ZK)=d+1.$
Clearly, $\sk_t\K=\sk_t(\sk_i\K)\subset \sk_t\L\subset\sk_t(\sk_j\K)=\sk_t\K$ for each $t\leq i.$ Hence $\sk_i\L=\sk_i\K$ and $\sk_1\L=\sk_1\K.$

The complex $\K$ is a flag complex, and $\sk_1\K=\sk_1\L,$ so $\K=\L^\mathrm{f}.$ By Corollary \ref{crl:nu_skeleton_bound}, we have $\nu(\L)\leq d-i.$ By Proposition \ref{prp:cat_zk_lower_bound}, $\cat(\ZK)\leq\cat(\Z_\L)+\nu(\L).$ This gives the inequality $\cat(\Z_\L)\geq (d+1)-(d-i)=i+1.$

On the other hand, $\cat(\Z_\L)\leq 1+\dim\L$ by \cite[Lemma 3.11]{beben_grbic}. Therefore, the upper bound $\cat(\Z_\L)\leq j+1$ follows from the inequality $\dim\L\leq\dim\sk_j\K=j.$
\end{proof}
Clearly, this corollary is true for any flag $\K$ such that $\dim\K=\max_{J\subset[m]}\cdim_\ZZ\K_J=d.$

\begin{exm} Every simple $n$-polytope $P$ corresponds to a triangulation $\partial P^*$ of $S^{n-1}.$ The \emph{moment-angle manifold} $\mathcal{Z}_P$ is the moment-angle complex $\mathcal{Z}_{\partial P^*}.$ Buchstaber and Limonchenko \cite{buchstaber_limonchenko} studied a certain family of flag polytopes $\mathcal{Q}=\{Q_n\}_{n\geq 0},$ $\dim Q_n=n.$ In particular, they calculated the LS-category of the corresponding moment-angle manifolds \cite[Theorem~6.11]{buchstaber_limonchenko}: $\cat(\Z_{Q_n})=n.$ To obtain a lower bound on $\cat(\Z_{Q_n}),$ they constructed a non-trivial $n$-product in $H^*(\Z_{Q_n})$ and used the well-known inequality $\cat(X)\geq\mathrm{cup}(X)$ between the LS-category and the \emph{cup-length} of a space $X.$ (We say that $\mathrm{cup}(X)\geq n$ if there are $\alpha_1,\dots,\alpha_n\in\H^*(X)$ such that $\alpha_1\smile\dots\smile\alpha_n\neq 0.$)

Clearly, Corollary \ref{crl:cat_manifold} gives the same value $\cat(\Z_{Q_n})=n$ by a different method. Also, by Proposition \ref{prp:mm_collapse}, the Milnor-Moore spectral sequence for $\Z_{Q_n}$ collapses at $E^2,$ not only at $E^{n+1}$ as shown in \cite[Theorem~6.11]{buchstaber_limonchenko}.
\end{exm}
\begin{prb}
\label{prb:cup_for_flag}
In general, the inequality $\cat(\ZK)\geq\mathrm{cup}(\ZK)$ can be strict (see \cite[Examples 5.8, 5.9]{beben_grbic}). However, $\K$ is not flag in these examples. This raises a question:
\begin{itemize}
\item Does $\cat(\ZK)=\mathrm{cup}(\ZK)$ hold for all flag simplicial complexes?
\end{itemize}

Since we know
the LS-category (Theorem \ref{thm:cat_final}) and
the cohomology ring \cite[Theorem 4.5.8]{ToricTopology} of moment-angle complexes in the flag case,
the problem has a purely combinatorial form:
\begin{itemize}
\item Let $\K$ be a flag simplicial complex on $[m].$ Set $d=\max_{J\subset[m]}\cdim_\ZZ\K_J.$ Does there always exist a collection of disjoint subsets $A_1,\dots,A_{d+1}\subset [m]$ such that the exterior product map
$$\H^0(\K_{A_1})\otimes\dots\otimes\H^0(\K_{A_{d+1}})\cong \H^{d}(\K_{A_1}\ast\dots\ast\K_{A_{d+1}})\to\H^{d}(\K_{A_1\sqcup\dots\sqcup A_{d+1}})$$ induced by $\K_{A_1\sqcup\dots\sqcup A_{d+1}}\hookrightarrow \K_{A_1}\ast\dots\ast\K_{A_{d+1}},$ is non-zero?
\end{itemize}
(The classes $\alpha_1,\dots,\alpha_{d+1}\in H^*(\ZK)$ can be chosen to be $(\ZZ\times\Zm)$-homogeneous, so $\alpha_i\in\H^{k_i}(\K_{A_i})$ for some disjoint $A_1,\dots,A_{d+1}\subset[m]$ and $k_i\geq 0.$ Then $k_i=0$ by dimension reasons.)

Passing from $\K$ to $\K_J,$ we can assume that $\max_{J\subset[m]}\cdim_\ZZ\K_J=\cdim_\ZZ\K$ and $\K=\K_{A_1\sqcup\dots\sqcup A_{d+1}}.$ Finally, we can assume that the classes $\alpha_i\in \H^0(\K_{A_i})$ are induced from the generator of $\H^0(S^0)$ by simplicial maps $\K_{A_i}\to S^0$ (every class in $\H^0(\K_{A_i})$ is a linear combination of such induced classes). This gives another equivalent form of the question:
\begin{itemize}
\item Let $\K$ be a flag simplicial complex and 
$d=\cdim_\ZZ\K.$
Let $C_d=S^0\ast\dots\ast S^0$ be the cross-polytope triangulation of $S^d.$ Does there always exist a full subcomplex $\K_J$ and a simplicial map $\K_J\to C_d$ such that $\H^d(C_d)\to\H^d(\K_J)$ is nonzero? 
\end{itemize}

Note that if $X$ is a finite simplicial complex with $\cdim_\ZZ X=d,$ then any $\alpha\in\H^d(X;\ZZ)$ can be induced from the generator of $\H^d(S^d;\ZZ)$ via some continuous map $f:X\to S^d$ (this follows from the obstruction theory). Taking a simplicial approximation, $f$ can be made simplicial for some subdivision of $X.$ So, in fact, we wonder whether flag simplicial complexes are ``subdivided enough''.
\end{prb}

%~

\section*{Acknowledgements}. The author would like to thank his advisor T.~E. Panov for his help, support and valuable advice, D.~I. Piontkovskii for his help with Koszul algebras, and the anonymous referee for very helpful comments and corrections. The author is deeply indebted to Kate Poldnik for her incomparable support and encouragement.

\section*{Funding}
This work was supported by the Theoretical Physics and Mathematics Advancement Foundation ``BASIS.'' The article was prepared within the framework of the HSE University Basic Research Program.

\end{document}